\documentclass[preprint]{elsarticle}
\usepackage[utf8]{inputenc}
\usepackage{graphicx}
\usepackage{tikz-cd}
\usepackage{amsmath}
\usepackage{amsfonts} 
\usepackage{mathrsfs}
\usepackage{amsthm}
\usepackage{relsize}
\usepackage{geometry}
\newgeometry{vmargin={14mm}, hmargin={18mm,16mm}}

\newtheorem{thm}{Theorem}
\newtheorem{lem}[thm]{Lemma}
\newdefinition{rmk}{Remark}
\newdefinition{exm}{Example}
\newtheorem{definition}{Definition}
\newproof{pf}{Proof}
\newcommand\scalemath[2]{\scalebox{#1}{\mbox{\ensuremath{\displaystyle #2}}}}

\title{Functors of Differential Calculus in Diolic Algebras}
\author[1]{Jacob Kryczka\corref{cor1}%
\fnref{fn1}}
\ead{jacob.kryczka@etud.univ-angers.fr}
\fntext[fn1]{email address: jacob.kryczka@etud.univ-angers.fr}
\address[1]{LAREMA, UMR 6093, du CNRS, Department of Mathematics, Université d'Angers, MathSTIC, France 
}
\date{June 2020}

\begin{document}

\begin{abstract}
We pose a new algebraic formalism for studying differential calculus in vector bundles. This is achieved by studying various functors of differential calculus over arbitrary graded commutative algebras (DCGCA) and applying this language to a particularly simple class of two-component graded objects introduced in this work, that we call \emph{Diolic algebras}. A salient feature of this \emph{conceptual} approach to calculus is that it recovers many well-known objects and notions from ordinary differential, symplectic and Poisson geometry but also provides some unique aspects, which are of their own independent interest.
\end{abstract}
\maketitle

\section{Introduction}
It is known that vector bundles can be studied as projective modules over smooth function algebras (\cite{Nes},\cite{Fat}) and in a new perspective posed in this work, a type of graded algebras. These algebras will be the fundamental objects of our study (definition \ref{DioleDef}) and are essentially $\mathbb{Z}$-graded commutative algebras $\mathscr{A}:=\mathscr{A}_0\oplus \mathscr{A}_1,$ with $\mathscr{A}_0:=A,$ a smooth commutative (unital) $\mathbb{K}$ algebra (in the sense of \cite{KraVerb}), $\mathscr{A}_1:=P$ an $A$-module and $\mathscr{A}_i:=\emptyset, i\neq 0,1.$ Multiplications in $\mathscr{A}$ are defined using the algebra structure  $\mathscr{A}_0\cdot\mathscr{A}_0=A\cdot A\subset\mathscr{A}_0$ on $A$, and the $A$-module structure in $P$, as $\mathscr{A}_0\cdot\mathscr{A}_1=A\cdot P\subset\mathscr{A}_1$, while all other products $\mathscr{A}_i\cdot\mathscr{A}_j$ are trivial. A graded algebra with these properties is what we refer to as a \emph{diole algebra}. 

In this setting, covariant derivatives of connections in $P$ appear as degree zero graded derivations of $\mathscr{A}$, geometrically known as Der-operators \cite{Fat}. In this way, we find an analogue of differential forms as skew-symmetric multi-linear functions on Der-operators, which lead to the so-called the Der-complex \cite{Rub01}. These types of complexes are important for physics and geometry since they provide an algebraic description of gauge algebras and gauge groups and because their cohomologies encode essential (global) data about internal structures supplied to vector bundles. In particular, Der-complexes contain information about the structure group and the corresponding cohomological invariants of projective modules, for instance Atiyah classes, Chern classes and Chern-Simons invariants.

Geometrically, a diole algebra is canonically associated with a given vector bundle $\pi:E\rightarrow M,$ over a smooth manifold. This diole algebra is defined as $\mathscr{A}=C^{\infty}(M)\oplus \Gamma(\pi),$ where $\Gamma(\pi)$ is the $C^{\infty}(M)$-module of smooth sections,
and thus contains all essential information about $\pi.$ In this way the study of vector bundles is equivalent to the study of differential calculus in the category of dioles. Therefore, the algebra of dioles and other graded algebras associated with a vector bundle (symmetric/exterior powers of the bundle etc.) will provide various analogs of the classical de Rham and Spencer complexes. The Der-complex is an example of this, being an analogue of degree $1$ of the de Rham complex.

Simplifications that come from this algebraic approach may be illustrated by theory of Lie algebroids. One may easily observe that a Lie algebroid is nothing but a degree $-1$ Poisson structure over an algebra of dioles (lemma \ref{DioleLieAlgbd}). In fact, when we have a Lie algebroid in the algebraic context, we may work in a more earnest way using the equivalent description as Lie-Rinehart algebras (see \cite{BruzzoLie}). In other words, a Lie-Rinehart algebra is a graded analogue of a Poisson manifold in the diolic category and this interpretation reveals new details not obvious in the standard geometric approach.
 
The efficacy of the diolic formalism is exemplified by allowing us to work directly with physical fields, in contrast to the modern approach to gauge theories which employ the heavy constructions of principle and associated bundles. This is due to the fact that Der-operators and the related calculus form a natural language to work with connections. For instance, a connection in a vector bundle $\pi:E\rightarrow M$ is algebraically interpreted as an $A$-module homomorphism  $\nabla:D(A)\rightarrow \mathrm{Der}(P),$ from the module of derivations of $A$, to the $A$-module of Der-operators in $P$
 with $A=C^{\infty}(M)$ and $P=\Gamma(\pi)$ (see \cite{Fat}).
This is physically very meaningful as, for instance, the gauge fields in Yang-Mills theory, electromagnetism etc. are suitable connections on principal bundles (or their associated bundles). One may employ the diole algebra approach and DCGCA to work directly with these fields and their connection-related differential calculus. This may yield a more satisfactory mathematical formalization of Yang-Mills theories and other non-abelian gauge theories. In this way we expect to reveal various important subtleties that are currently hidden in the geometric approach based on principal bundles and anticipate various new instruments for studying Yang-Mills theories.

\subsection{The Hamiltonian formalism over graded algbras}
One of the physical aspects of the proposed formalism is the diolic generalization of the standard algebraic Hamiltonian formalism of classical mechanics \cite{KraVin}, whose arena of action is the cotangent space $T^*M$ (the phase space), over a smooth manifold (the configuration space) $M$.
These generalizations are made possible by employing the algebraic theory of differential operators which is based on the transition from standard differential calculus over smooth manifolds to differential calculus over graded commutative algebras. Namely, these generalizations are based on the fact that when passing from the `geometric' picture to the `algebraic' we find that the \textit{algebra of symbols} $\text{Smbl}_
*(A)=\bigoplus_{n\geq 0}\text{Smbl}_n(A),$ where $\text{Smbl}_n(A):=\text{Diff}_n(A)/\text{Diff}_{n-1}(A)$ as discussed in \cite{KraVin},
is a good substitute for the algebra of functions on $T^*M$. In particular, the algebra $\text{Smbl}_*(A)$ naturally carries a Lie algebra structure with respect to the bracket $\{\text{smbl}_k(\Delta),\text{smbl}_{\ell}(\nabla)\}:=\text{smbl}_{k+\ell-1}\big([\Delta,\nabla]\big),$ where $\text{smbl}_i(\Box)\in\text{Smbl}_i(A)$ is the main symbol of the operator $\Box.$ The so defined bracket determines a derivation,  $X_{\mathfrak{s}}:=\{\mathfrak{s},.\}$ of the algebra $\text{Smbl}_*(A),$ for some fixed symbol $\mathfrak{s}$, and $X_{\mathfrak{s}}$ is interpreted as a Hamiltonian vector field on $\text{Spec}_{\mathbb{R}}\big(\text{Smbl}_*(A)\big).$

Geometrically,
$\text{Spec}_{\mathbb{R}}\big(\text{Smbl}_* (C^{\infty}(M))\big)$ is canonically identified with the total space of the
cotangent bundle $T^*M$, (see, for instance \cite{Nes}) with the above bracket coinciding with the standard Poisson bracket on $T^*M$. 
Derivations $X_{\mathfrak{s}}$ in this context are Hamiltonian vector fields in the usual sense. In this way we can develop analogues of the Hamiltonian formalism over graded commutative algebras by interpreting the algebra $\text{Smbl}_{*}(A)$ to be the algebra of observables of coordinates/momenta of the considered system. This algebraic interpretation allows to develop the full Hamiltonian formalism over graded commutative algebras, for instance over super-manifolds, just by passing to notions of graded derivations, graded differential operators, etc.

Graded Lie algebras play a fundamental role in this language not only because of their importance in the context of super-symmetry, but also because they are adequate for many models of internal structures of fields, particles and other physical notions. A combination of differential calculus with these graded Lie algebras automatically respects these internal structures.
 So, the corresponding algebraic Hamiltonian formalism does as well. 

The above algebraic interpretation of the Hamiltonian formalism also reveals important relations to the theory of propagations of fold-type singularities of partial differential equations. In this framework solution singularities of an equation $\mathscr{E}$ are described by a system of differential equations, one of which is the Hamilton-Jacobi equation associated with the main symbol of $\mathscr{E}$. This is well known in geometrical optics and as it was already observed by Schrödinger, is intimately related to the \textit{quantization problem} (see \cite{Vin02,Vin03,Vin04}). Our expectation is that this optics-mechanics analogy, if considered in the context of a suitable graded algebra, will lead to a more satisfactory approach to the quantization problem. More exactly, in this graded framework we may investigate and address the question: "what are such singularities related with the quantization problem?" 
The current geometric theory of propagation of singularities doesn't take into consideration the internal structures which are most interesting to physics and we expect that by using a graded approach we can describe more interesting singularities which respect these internal structures.
Achieving this description demands a generalization from the algebra of dioles to a related type of algebra called a \emph{triolic algebra}, but this is beyond the scope of the current paper and is a focus of our future work.

\subsection{Dioles, trioles and general algebraic extensions}

The notion of a trivial extension of a ring $R$ over a field $\mathbb{K}$ of characteristic zero, by an $R$-module $P$ has been used extensively in many areas of mathematics; it is ubiquitous in Ring theory, Homological algebra, Representation theory and Category theory, with further applications to higher categorical formulations of geometry (see for instance \cite{PorVez}). This extension $R[P]$, is defined to have an underlying additive group $R\oplus P$ with multiplication law
$(r_1,p_1)\cdot (r_2,p_2):=(r_1r_2,r_1p_2+r_2p_1),$ for $r_1,r_2\in R,p_1,p_2\in P.$
We will be interested in the situation when $R$ is a commutative unital $\mathbb{K}$-algebra, $A.$

\begin{exm}
Let $P$ be an $A$-module and consider the symmetric algebra of $P$ over $A$, $\text{Sym}_A^*(P)$. Since $\text{Sym}_A^0(P)=A,\text{Sym}_A^1(P)=P,$ we have an isomorphism of graded $A$-algebras,
$\text{Sym}_A^*(P)/\bigoplus_{n\geq 2}\text{Sym}_A^n(P)\cong A\oplus P.$
\end{exm}

\begin{exm}
Let $P$ be a free $A$-module with a basis $V.$ Then $A\oplus P$ is isomorphic to $A[X_1,..,X_V]/\{X_v\}_{v\in V}^2,$ where $X_i$ are some indeterminates labelled by elements of the basis.
In particular, $A\oplus A\cong A[X]/X^2$ is the ring of dual numbers of $A.$ 
\end{exm}

In this paper we re-examine the notion of a trivial extension and study it in a different context. A priori, there is no reason to stop at square-zero extensions and one can generalize to extensions of arbitrary length. Without entering into details, such an extension of length $n$, is given, for some family of $A$-modules $P_1,..,P_n$ endowed with $A$-bilinear maps $g=\{g_{i,j}:P_i\otimes_AP_j\rightarrow P_{i+j}\},$ to be the object whose underlying abelian group is $A\oplus P_I:=A\oplus P_1\oplus..\oplus P_n$ with multiplication given by 
$(p_0,p_1,..,p_n)\cdot (\tilde{p}_0,\tilde{p}_1,..,\tilde{p}_n)=\sum_{j+k=i}g_{j,k}\big(p_j,\tilde{p}_k\big).$
Defined in this way,
 $A\oplus P_I$ is a commutative algebra with unit $1_{A\oplus P_I}=(1_A,0,0...,0).$ More than this, it is a sub-ring of the formal triangular matrix ring,
$$
 \scalemath{0.8}{
\begin{pmatrix}
A & P_1 & P_2 & ...& P_n
\\
0 & A & P_1 & ...  & P_{n-1}
\\
. & . & ...& . & .
\\
0 & 0 & ...& ...& A
\end{pmatrix}}.
$$

\begin{rmk}
Suppose $A$ is a field $\mathbb{K}$ and $P_i=\mathbb{K}$ for all $i$. Then such families of formal triangular matrices coincide with upper-triangular Toeplitz matrices.
\end{rmk}

The novel objects introduced in this paper provide a re-imagination of a square zero extension in the language of graded commutative algebra, where we have renamed them \emph{diolic algebras}, to emphasize our change in viewpoint which is suitable for our purposes. Interesting features appear at the next logical generalization, what we call triolic algebras. The latter are related to $2$-trivial square extensions in an entirely analogous manner to which dioles relate with square zero extensions. Similar definitions arise for $n$-olic algebras and these generalizations are under investigation in a sequel work.

\subsubsection{The untold story of square zero extensions}
Objects of the form $R\oplus P$ and more recently, their $n$-trivial generalizations $R\oplus P_I$ have played an important role in numerous areas of mathematics; however, one large part of their story remains $-$ their \emph{differential calculus}, in the sense of \cite{Vin01},\cite{KraVerb},\cite{Ver}.
The main purpose of this work is to study the representability of the various functors of
algebraic calculus but over arbitrary \emph{graded commutative} algebras as in \cite{Ver}, \cite{KerKra01} and \cite{KerKra02}, 
where we follow closely the notions therein and apply these constructions to a particular class of two component graded algebras. In doing so, we achieve a self-contained algebraic Calculus over $1$-trivial extensions. This is achieved by viewing the datum of a square zero extension of a smooth geometric commutative algebra $A$ by a geometric $A$-module (in the sense of \cite{Nes}) \emph{not} as some "infinitesimally thickened" algebra $A\oplus P$, as is traditionally done $-$ for instance in algebraic geometry $-$ but entirely as their own species of mathematical creatures.

The formalism developed in this paper provides new insight into the differential geometry of vector bundles as aspects of differential Calculus over a single graded algebra. We believe this approach is rather satisfactory for it is grounded in the conceptual nature of applying the logic algebra of differential calculus of \cite{Vin01}, over diolic algebras.

More than this, we see that being diolic as a mathematical object is surprisingly ubiquitous. Indeed, Lie algebras of diolic type and Poisson algebras of diolic type in the category of $\mathbb{K}$-modules, contain the information of Lie algebra representations and Lie algebroids, respectively. 
Moreover, by understanding a Poisson vertex algebra as a Poisson algebra object in the category of $\mathbb{K}[\partial]$-modules, one finds that a Poisson vertex algebra of diolic type is precisely a Courant-Dorfman algebra and when the underlying bilinear form is invariant, we recover Courant algebroids. It is our expectation that similar interesting manifestations of usual mathematical structures will occur as we further investigate $n$-olic algebras.

\subsection{Outline of the paper}
In what follows all constructions are algebraic, which is to say they take place over arbitrary (graded) modules over (graded) algebras, but this enables a free passage to a geometric setting by implementing the $C^{\infty}$-Swan theorem \cite{Nes}. The content of this paper is organized as follows.
In Section 2 we collect some facts about differential calculus in arbitrary graded commutative algebras. In Section 3, we introduce the category of diolic algebras and modules over them. In Section 4, we describe the functors of graded calculus over diolic algebras, paying special attention to Poisson structures. 

The graded perspective we adopt on diole algebras results in the corresponding differential calculus exhibiting the following significant features when our diole algebra is the canonical one associated with a vector bundle $\pi.$ On one hand, it describes a decomposition of our objects of calculus (ie. derivations, differential operators etc) into those operators which act in the total space of the vector bundle and which project surjectively onto the base manifold, which is to say that they possess a meaningful notion of symbol map, and into a vector-valued calculus.
On the other hand, we find that the functors of calculus over a diole algebra admit non-trivial descriptions in negative degrees if and only if our vector bundle is of rank $1.$ 
Consequently, our diolic formalism unites and generalizes familiar features of vector bundle calculus as being part of the calculus over a single object. 
In the case of a line bundle the differential calculus over diole algebras yields a conceptual home for studying Jacobi geometry and related structures, for instance Jacobi structures and Jacobi algebroids.
As a particular case of our constructions, we recover the Atiyah sequence of a vector bundle and we indicate how our formalism naturally suggests a generalization of this sequence to include differential operators of higher degree $k>1.$
In Section 5, we describe some aspects of the algebraic Hamiltonian formalism (in the sense of \cite{KraVin}) over diolic algebras by characterizing the symbols of diolic differential operators and providing a local description of the corresponding diolic Poisson bracket.

\section*{Acknowledgements}
This work contains a part of the authors PhD thesis in preparation at LAREMA, Department of Mathematics.
This program is supported by the Government of Canada student aid service (NSLS). It is my pleasure to thank these organizations. 

The subject of this thesis work was elaborated and proposed by Professor Alexandre Vinogradov, who passed away on 20 September 2019.
I wish to express my deepest gratitude for his willingness to engage in discussions during the early stages of this work. Professor Vinogradov was a positive figure who played a very significant role in the development of my scientific interests. It is very meaningful and of great importance to me that I had the opportunity to interact with him on a personal and scientific level for many years. It is my sincerest wish to dedicate this work to him. 

I would like to thank my advisor, Professor Volodya Rubtsov for his encouragement, guidance and endless patience during the preparation of this paper.
A part of this work was completed during the authors visit to the University of Salerno, and the author would like to thank Professor Luca Vitagliano for his efforts in facilitating this work, both during this visit and during the many Diffiety Institute summer schools. It is my pleasure to also thank: Dr. Denys Dutykh of CNRS LAMA, University of Chambery, for taking the time to read the early version of this work and providing detailed feedback, as well as for showing general interest and support in my academic endeavours, Dr. Geoffrey Powell of LAREMA for useful discussions and my thesis committee members, Professor Joseph Krasil'shchik of Moscow Independent University and Dr. Vladimir Salnikov of CNRS La Rochelle, to whom I am indebted for their willingness to oversee this doctoral project.

\subsubsection{Conventions}
\label{ssec:conventions}
Let $(\mathrm{C},\otimes)$ be a symmetric monoidal, closed and (co)-complete category and let $G$ be an abelian group. Let $\mathrm{C}^{G}:=\prod_{g\in G}\mathrm{C}$ be the category of $G$-graded objects in $\mathrm{C},$ with $\mathrm{gcAlg}(C):=\mathrm{cAlg}(\mathrm{C}^{G})$ the corresponding commutative algebra objects in $\mathrm{C}^G$, or simply $G$-graded commutative algebras in $\mathrm{C}.$ In particular, we will consider the category $\mathrm{Vect}_{\mathbb{K}}$ of vector spaces over a ground field $\mathbb{K}$ of characteristic zero. A vector space $V$ is said to be $G$-\emph{graded} if $V=\bigoplus_{g\in G} V_g$ with $V_g\in\mathrm{Vect}_{\mathbb{K}},$ with $V_0:=\mathbb{K}.$
Elements $v\in V_g$ are said to be \emph{homogeneous of degree} $g\in G.$
A vector subspace $W\subset V$ is a \emph{graded sub-space} if $W=\bigoplus_{g\in G} W_g$ with $W_g=W\cap V_g.$ Given such a subspace, the quotient $V/W=\bigoplus_{g\in G}V_g/W_g$ is also $G$-graded.
The natural operations of tensor product and internal homomorphism in $\mathrm{Vect}_{\mathbb{K}}$ also inherit a well-defined notion of grading.
Namely for two $G$-graded vector spaces $V,U$ the we have $V\otimes_{\mathbb{K}}U$ is $G$-graded with homogeneous degree $(V\otimes_{\mathbb{K}}U)_g=\bigoplus_{h+f=g}V_h\otimes_{\mathbb{K}}U_f.$ Similarly, $\text{Hom}_{\mathbb{K}}(V,U)$ obtains a $G$-grading with homogeneous degree $g$ piece $\text{Hom}_{\mathbb{K}}(V,U)_g:=\{f:V\rightarrow U| \mathbb{K}-\text{linear } f(V_h)\subset U_{g+h}, h\in G\}.$
We also have a well-defined graded binary co-product $V\oplus U$ with $(V\oplus U)_g:=V_g\oplus U_g.$
A $G$-graded vector space $A$ is said to be a $G$-\emph{graded algebra} if it is endowed with a multiplication map $A_g\cdot A_h\subset A_{g+h},$ for $g,h\in G.$
If this graded algebra is unital, the unit is necessarily in $A_0.$
Consider a pairing $\mu:G\times G\rightarrow \mathbb{K}/\{0\},$ sending $(g_1,g_2)$ to $\mu(g_1,g_2)$ such that $\mu(g_1,g_2)^{-1}=\mu(g_2,g_1)$ and $\mu(g_1+g_2,g_3)=\mu(g_1,g_3)\mu(g_2,g_3).$
Such a map is called a \emph{sign form} or a \emph{commutation factor}. From the definition it readily follows that $\mu(g_1,g_2+g_3)=\mu(g_1,g_2)\mu(g_1,g_3)$ and $\mu(g,0)=\mu(0,g)=1.$
A pair consisting of an additive abelian group $G$ and a commutation factor is called a \emph{grading group}, $\mathcal{G}=(G,\mu).$
A $\mathcal{G}$-graded algebra $A$ is \emph{graded commutative} if the multiplication operation obeys the \emph{Koszul sign rule}
\begin{equation}
    a\cdot b=\mu(a,b)b\cdot a.
\end{equation}
For the remainder of our paper, we will care about a particular simple commutation factor, which is also the standard one familiar to graded or super-geometry. Namely, when $G=\mathbb{Z}$ or $\mathbb{Z}_2,$ there is only one nontrivial commutation factor-the super-commutation factor $\mu(g_1,g_2):=(-1)^{g_1g_2}.$ 

We denote by $\mathrm{Alg}(\mathbb{K})$ the category of $\mathbb{K}$-algebras, and $\text{gcAlg}(\mathbb{K})$ the  category of graded commutative unital $\mathbb{K}$-algebras.
Given a graded $\mathbb{K}$-algebra $A$, a $G$-\emph{graded} $A$ \emph{module} $P$ is a graded $\mathbb{K}$-vector space $P=\bigoplus_{g\in G} P_g$ with an action by $A,$ denoted homogeneously as $A_g\cdot P_h\subset P_{g+h}.$ The category of left (resp. right) graded $A$-modules is denoted $\mathrm{gMod}^<(A)$ (resp. $\mathrm{gMod}^>(A)$).

\section{Differential calculus over graded algebras}
\label{sec:Calc}
We provide some preliminaries on the functors of graded calculus \cite{Vin01}, referring the reader to \cite{Kra01}, \cite{KraVerb} for details in the un graded setting and to \cite{KerKra01},\cite{KerKra02},\cite{Ver} for the graded setting.
\begin{definition}
\label{gradedderivationdefinition}
Let $\mathcal{R}$ be a ring and let $\mathscr{A}$ be a $\mathcal{G}$-graded commutative $\mathcal{R}$-algebra and let $\mathscr{P}$ be a left $\mathscr{A}$-module. A graded $\mathcal{R}$-module homomorphism 
$X:\mathscr{A}\rightarrow \mathscr{P}$ is a \emph{degree k derivation} of $\mathscr{A}$ with values in $\mathscr{P}$ if it satisfies the following graded Leibniz rule,
$$X(ab)=X(a)b+\mu(a,X)aX(b)\equiv \mu(X+a,b)bX(a)+\mu(a,X)aX(b),$$
for all $a,b\in\mathscr{A}.$
\end{definition}
\begin{rmk}
When our grading group is $\mathbb{Z}$ or $\mathbb{Z}_2$, then $\mu(a,b)=(-1)^{a\cdot b}$ is the usual commutation factor and definition \ref{gradedderivationdefinition} is the usual one. We write all commutation factors as $(-1)^{a\cdot b}$ in the following.
\end{rmk}
Denoting the collection of such objects as $D(\mathscr{P})_{\mathcal{G}}:=D(\mathscr{A},\mathscr{P})_{\mathcal{G}},$ we can study its categorical properties. 
\begin{lem}
The assignment $\mathscr{P}\longmapsto D(\mathscr{A},\mathscr{P})_{\mathcal{G}},$ defines a functor
$D(\mathscr{A},-)_{\mathcal{G}}:\mathrm{gMod}^{\mathcal{G}}(\mathscr{A})\rightarrow \mathrm{gMod}^{\mathcal{G}}(\mathscr{A}).$
\end{lem}

Equipped with the graded commutator defined in the usual way, the $\mathcal{G}$-graded module $D(\mathscr{A})$ is a graded Lie algebra.

Moreover, this functor is representable and the representative objects are graded differential forms.
\begin{thm}
\label{RepofDers}
Let $\mathscr{A}\in\mathrm{gcAlg}(\mathbb{K}),$ and let $\mathscr{P}\in\mathrm{gMod}^{\mathcal{G}}(\mathscr{A}),$ be a graded left $\mathscr{A}$-module. 
There is a universal object $\Lambda^1(\mathscr{A})$ equipped with a derivation of degree zero, $d$, such that for any graded derivation $\Delta \in D_{\mathcal{G}}(\mathscr{P}),$ there is a unique graded module homomorphism $h^{\Delta},$ making the diagram
\[
\begin{tikzcd}
\mathscr{A}\arrow[r, "\Delta"] \arrow[dr,"d"] & \mathscr{P} \\
&  \Lambda^1(\mathscr{A})\arrow[u, "h^{\Delta}"]
\end{tikzcd}
\]
commute.
\end{thm}
One may verify the existence of such a universal graded module by considering the free graded module generated by symbols $\{\delta_a:a\in\mathscr{A}\},$ denoted by $\tilde{\Lambda}(\mathscr{A}).$ The grading is inherited from that on $\mathscr{A}$ as $\delta_a\in\tilde{\Lambda}(\mathscr{A})^i,$ when $a\in\mathscr{A}^i.$ Consider the ideal $\mathscr{I}_0,$ generated by elements $\delta_a$ subject to the relations,
\begin{equation*}
\mathscr{I}_0:\begin{cases}
\delta_{a+b}-\delta_a-\delta_b\\
\delta_{\lambda b}-\lambda\delta_b\\
\delta_{ab}-a\delta_b-(-1)^{a\cdot b}b\delta_a,\\
\end{cases}
\end{equation*}
for any homogeneous elements $a,b\in\mathscr{A}$ and for any $\lambda \in k.$ 
Consider quotient module $\Lambda^1(\mathscr{A}):=\tilde{\Lambda}^1(\mathscr{A})/\mathscr{I}_0,$ endowed with the natural projection $\tilde{\pi}:\tilde{\Lambda}(\mathscr{A})\rightarrow\tilde{\Lambda}(\mathscr{A})/\mathscr{I}_0.$ We then set $d(a):=\tilde{\pi}(\delta_a),$ for every $a\in\mathscr{A}.$ It now remains to demonstrate that $\Lambda^1(\mathscr{A})$ has a well-defined graded $\mathscr{A}$-module structure, that $d$ is a derivation of degree $0$ and that for any graded derivation $\Delta$ there is a unique graded homomorphism satisfying the above universal property. These details are checked in a rather straightforward manner.

Consequently any derivation $\Delta$ is represented uniquely as $\Delta=h^{\Delta}\circ d$. The $\mathscr{A}$-module $\Lambda^1(\mathscr{A})$ is generated by the image of $d$. Since $d$ is a derivation of degree $0$ it follows that $\Lambda^1(\mathscr{A})=\bigoplus_g\Lambda^1(\mathscr{A}_g)=\bigoplus_g\Lambda^1(\mathscr{A})_g.$ Moreover, for any $i\geq 0$ we define $\Lambda^i(\mathscr{A}):=\Lambda^1(\mathscr{A})\wedge..\wedge \Lambda^1(\mathscr{A})$ ($i$ times). Here the wedge must be understood in the graded sense of \cite{Ver}.
The universal differential $d$ extends to the entire de Rham algebra as a graded derivation with respect to this product by setting $d(da):=0$ for any $a\in \mathscr{A}$ and $d(\omega\wedge \theta):=d\omega \wedge \theta +(-1)^{p}\omega\wedge d\theta,$ where $\omega\in \Lambda^p(\mathscr{A}).$ Consequently, as in the ungraded situation we obtain the \emph{graded de Rham complex} of the $\mathcal{G}$-graded algebra $\mathscr{A},$  
\begin{equation}
0\rightarrow\mathscr{A}\xrightarrow{d_{\mathscr{A}}}\Lambda^1(\mathscr{A})\rightarrow...\rightarrow \Lambda^i(\mathscr{A})\xrightarrow{d_{\mathscr{A}}}\Lambda^{i+1}(\mathscr{A})\rightarrow...,
\end{equation}
whose cohomology is a $\mathcal{G}\oplus\mathbb{Z}$-graded module denoted by  $H_{dR}(\mathscr{A}).$

By restricting the differential $d_{\mathscr{A}}^0$ to a homogeneous component, say $g\in G,$ and extending linearly by the graded Leibniz rule, we may specify a homogeneous de Rham complex, called the de Rham complex of $\mathscr{A}$ of degree $g.$

The next functor of algebraic differential calculus that we require is the functor of differential operators.
 Recall that
a degree $g$ homomorphism of graded $\mathcal{R}$ modules is an $\mathcal{R}$-module homomorphism
$\varphi\colon\mathscr{P}\rightarrow\mathscr{Q},$
such that
$\varphi\big(\mathscr{P}_{h}\big)\subseteq \mathscr{Q}_{g+h},$
for all $h\in G.$ We write 
the $\mathcal{R}$-module of degree $g$ morphisms as
$\text{Hom}_{\mathcal{R}}^g\big(\mathscr{P},\mathscr{Q}\big),$ and the totality of graded morphisms as 
$\text{Hom}_{\mathcal{R}}^{\mathcal{G}}\big(\mathscr{P},\mathscr{Q}\big):=\bigoplus_{g\in G}\text{Hom}_{\mathcal{R}}^g\big(\mathscr{P},\mathscr{Q}\big).$ These
constitute a $G$-graded $\mathcal{R}$-module.

In $\text{Hom}_{\mathcal{R}}^{\mathcal{G}}\big(\mathscr{P},\mathscr{Q}\big),$ we have two graded $\mathscr{A}$-module structures. These are the left and right multiplications, denoted by $<,>$ respectively, which are defined as 
\begin{eqnarray*}
a^<\varphi&:=&a_{\mathscr{Q}}\circ\varphi,
\\
a^>\varphi&:=&(-1)^{a\cdot \varphi}\varphi\circ a_{\mathscr{P}},
\end{eqnarray*}
where $a_{\mathscr{P}},a_{\mathscr{Q}}$ denote the endomorphisms given by multiplication by $a$ of $\mathscr{P},\mathscr{Q}.$ For instance $a_{\mathscr{P}}:\mathscr{P}\rightarrow \mathscr{P}, p\mapsto a_{\mathscr{P}}(p):=ap$ and we understand that this endomorphism shifts the degree by $\deg(a).$

For a given $a\in\mathscr{A}$ of homogeneous degree, we set for an arbitrary $\varphi\in\text{Hom}_{\mathcal{R}}^{\mathcal{G}}(\mathscr{P},\mathscr{Q}),$ 
\begin{equation}
\delta_a\varphi:=\big[a,\Delta]=a^<\varphi-a^>\varphi,
\end{equation}
where $[-,-]$ denotes the graded commutator
$[a,b]=a\circ b-(-1)^{a\cdot b}b\circ a,$
for any $a,b$ graded elements. Using this definition, one can see that $\delta_a$ defines a morphism 
$\delta_a:\text{Hom}_{\mathcal{R}}^{\mathcal{G}}\big(\mathscr{P},\mathscr{Q}\big)\rightarrow \text{Hom}_{\mathcal{R}}^{\mathcal{G}}\big(\mathscr{P},\mathscr{Q}\big),$ of degree $\deg(a).$ We may can consider nested graded commutators and set
$$\delta_{a_0,a_1,..,a_k}:=\delta_{a_0}\circ \delta_{a_1}\circ...\circ\delta_{a_k},$$
for any $a_0,..,a_k\in\mathscr{A}$, homogeneous elements. It is not difficult to see that these satisfy the relation $\delta_{ab}=a\delta_b+(-1)^{ab}\delta_ab,$ for $a,b\in\mathscr{A}.$

\begin{definition}
\label{diffopdef}
Let $\nabla\in \mathrm{Hom}_{\mathcal{R}}^{\mathcal{G}}(\mathscr{P},\mathscr{Q}).$ Then $\nabla$ is said to be a \emph{graded linear differential operator} of order $\leq k$, if 
$\delta_{a_0,a_1,..,a_k}(\nabla)=0,$
for all $a_0,a_1,..,a_k\in\mathscr{A}.$
\end{definition}
The set of all graded differential operators of order $\leq k,$ of homogeneous degree $g\in G$ is denoted by 
$\text{Diff}_{k}(\mathscr{P},\mathscr{Q})_g,$ with the totality of such order $k$ operators denoted by 
$\text{Diff}_{k}(\mathscr{P},\mathscr{Q})_{\mathcal{G}}:=\bigoplus_{g\in G}\text{Diff}_{k}(\mathscr{P},\mathscr{Q})_g.$
This is closed with respect to both module structures and is therefore a graded $\mathscr{A}$-bimodule\footnote{When the reference to the underlying left, right or bi-module structures is not relevant, we will simply write $\text{Diff}_{k}(-,-)_{\mathcal{G}}$.}, denoted by $\text{Diff}_{k}^{>,<}(\mathscr{P},\mathscr{Q})_{\mathcal{G}},$ for all $k\geq 0.$

Analogous to the non-graded situation we have the standard order filtration, defined by the following chain of natural inclusions,
\begin{equation}
\label{eqn:orderfil}
\text{Diff}_{0,\mathcal{G}}\subset \text{Diff}_{1,\mathcal{G}}\subset..\subset\text{Diff}_{k,\mathcal{G}}\subset \text{Diff}_{k+1,\mathcal{G}}\subset...
\end{equation}
We therefore may consider the (filtered) module of \textit{all} graded differential operators, as the direct limit of this sequence $
\text{Diff}(\mathscr{A},\mathscr{A})_{\mathcal{G}}:=\bigcup_{k\geq 0}\text{Diff}_{k}(\mathscr{A},\mathscr{A})_{\mathcal{G}}.$
The associated graded algebra with respect to the filtration above, is called the algebra of \emph{symbols} (see below for details).
Such a colimit is closed with respect to the graded commutator of graded differential operators
$$[-,-]:\text{Diff}_{k}(\mathscr{A},\mathscr{A})_g\otimes_{\mathbb{K}}\text{Diff}_{\ell}(\mathscr{A},\mathscr{A})_h\rightarrow \text{Diff}_{k+\ell-1}(\mathscr{A},\mathscr{A})_{g+h},$$
which is defined in the usual way on homogeneous elements
$[\Delta_g,\Delta_h]:=\Delta_g\circ \Delta_h-(-1)^{g\cdot h}\Delta_h\circ\Delta_g.$
\begin{lem}
We let $\chi^n$ be the ordered subset of $n$ natural numbers and let $\chi=(i_1,..,i_{\ell})$ be an ordered subset, meaning $i_1<i_2<..<i_{\ell}$ with $\ell\leq n.$ We let $\bar{\chi}$ denote the ordered compliment. Moreover, we write $|\chi|=\ell,$ for the length.
We then have
$$\delta_{a_{\chi^n}}(\Delta)(b)=\sum_{|\chi|\leq n}(-1)^{|\chi|+<a_{\bar{\chi}},a_{\chi}>+\Delta\cdot a_{\bar{\chi}}}a_{\chi}\cdot \Delta(a_{\bar{\chi}}\cdot b).$$
\end{lem}
If our differential operator is of order $m<n,$ the above expression reduces to $$\Delta(a_{\chi^n}b)=-\sum_{0<|\chi|\leq n}(-1)^{|\chi|+<a_{\bar{\chi}},a_{\chi}>+\Delta\cdot a_{\chi}}a_{\chi}\cdot \Delta(a_{\bar{\chi}}b).$$
Note that by using these results, in particular, the relation
$\delta_a(\Delta\circ\nabla)=\delta_a(\Delta)\circ\nabla+(-1)^{a\cdot \nabla}\Delta\circ\delta_a(\nabla),$ we obtain a map
$\text{Diff}_{s}^{>,<}(\mathscr{A},\mathscr{B})_g\otimes_{\mathbb{K}}\text{Diff}_{t}^{>,<}(\mathscr{B},\mathscr{C})_h\rightarrow \text{Diff}_{s+t}^{>,<}(\mathscr{A},\mathscr{C})_{g+h},$
which is precisely the composition of differential operators. 
\begin{rmk}
One should establish this fact using some care in a more categorical fashion, via the so-called gluing morphisms of \cite{KraVerb} adapted to the graded setting. These are suitable graded natural transformations and the discussion of these is treated carefully in \cite{Kry}.
\end{rmk}
The associated graded object to the filtration (\ref{eqn:orderfil}) is now discussed. For each $k\geq 0,$ one sets $\text{Smbl}_{k}:=\text{Diff}_k/\text{Diff}_{k-1},$ with $\text{Diff}_{-1}:=\emptyset,$ 
and the entire associated graded object
$$\text{Smbl}(\mathscr{A})_{\mathcal{G}}:=\bigoplus_{k\geq 0}\text{Smbl}_{k}(\mathscr{A},\mathscr{A})_{\mathcal{G}}\equiv \bigoplus_{k\geq 0}\bigoplus_{g\in\mathcal{G}}\text{Smbl}_{k}(\mathscr{A},\mathscr{A})_g,$$
a $\mathcal{G}\oplus\mathbb{Z}$-graded commutative associative algebra, is the algebra of \textit{graded symbols}.
For $\Delta_g \in \text{Diff}_{k}(\mathscr{A})_g,$ denote the image of the projection
$\text{smbl}_{k,g}:\text{Diff}_{k}(\mathscr{A})_g\rightarrow \frac{\text{Diff}_{k}(\mathscr{A})_g}{\text{Diff}_{k-1}(\mathscr{A})_g}=\text{Smbl}_{k}(\mathscr{A})_g,$
by
$\text{smbl}_{k,g}(\Delta_g):=\mathfrak{s}_{k,g}(\Delta_g).$

\begin{rmk}
The quotient homomorphism above is the restriction of $\text{smbl}_{\bullet,*}$ to the $k$'th filtered component of the order filtration and $g$'th homogeneous graded component,
$\text{smbl}_{k,g}=\text{smbl}_{\bullet,*}|_{\mathrm{Diff}_{k}(\mathscr{A})_g}.$
\end{rmk}
The commutative algebra structure is given by the map
\begin{equation}
\label{symbolalgebrastructure}
\star:\text{Smbl}_{k}(\mathscr{A})_i\times \text{Smbl}_{\ell}(\mathscr{A})_j\rightarrow \text{Smbl}_{k+\ell}(\mathscr{A})_{i+j}, \hspace{2mm} \text{smbl}_{k,i}(\Delta)\star\text{smbl}_{\ell,j}(\nabla):=\text{smbl}_{k+\ell,i+j}\big(\Delta\circ\nabla)
\end{equation}
for $\Delta\in \text{Diff}_{k}(\mathscr{A})_i,\nabla\in \text{Diff}_{\ell}(\mathscr{A})_j.$
The graded commutativity of the graded associative algebra $\text{Smbl}_{\bullet,*}(\mathscr{A},\mathscr{A}),$ is expressed in 
 $\mathfrak{s}_{k,i}(\Delta)\star\mathfrak{s}_{\ell,j}(\nabla)=(-1)^{i\cdot j}\mathfrak{s}_{\ell,j}(\nabla)\star\mathfrak{s}_{k,i}(\Delta),$
for symbols of graded differential operators $\Delta,\nabla$ of orders $k,\ell$ and graded degrees $i,j$, respectively. This follows from
$\delta_{a_{\chi^n}}(\Delta\circ\nabla)=\sum_{|\chi|\leq n}(-1)^{\nabla\cdot a_{\chi}}\delta_{a_{\chi}}(\Delta)\circ\delta_{a_{\bar{\chi}}}(\nabla),$
  where $|a_{\chi}|=\sum_{j=1}^{\ell}\deg(a_{i_{j}}).$ 

\begin{lem}
\label{symbolalg}
Let $\mathscr{A}$ be a graded commutative algebra. Then the collection of graded symbols $\mathrm{Smbl}_{\bullet,*}(\mathscr{A})$ form a graded commutative algebra a graded Lie algebra with respect to the bracket,
$$\{\mathrm{smbl}_{k,i}(\Delta),\mathrm{smbl}_{\ell,j}(\nabla)\}:=\mathrm{smbl}_{k+\ell-1,i+j}\big([\Delta,\nabla]\big),$$
where $\mathrm{smbl}_{k,i}(\Delta)\in\mathrm{Smbl}_{k,i}(\mathscr{A}),\mathrm{smbl}_{\ell,j}(\nabla)\in\mathrm{Smbl}_{\ell,j}(\mathscr{A}).$ Moreover, this bracket is a graded Poisson bracket.
\end{lem}

The utility of this bracket is apparent. For instance, we can construct a suitable notion of a  Hamiltonian derivation in the graded formalism as follows.
 Let $\Delta\in \text{Diff}_{k}(\mathscr{A})_i,$ be some graded differential operator of order $k$ and degree $i$ with symbol $\text{smbl}_{k,i}(\Delta):=\mathfrak{s}_{k,i}(\Delta).$ We may define a map in terms of the above graded Poisson bracket as
\begin{equation*}
H_{\mathfrak{s}_{k,i}(\Delta)}:\text{Smbl}_{\bullet,*}(\mathscr{A})\rightarrow \text{Smbl}_{\bullet,*}(\mathscr{A}),\hspace{2mm}
\mathfrak{s}_{m,j}(\nabla) \longmapsto \big\{\mathfrak{s}_{k,i}(\Delta),\mathfrak{s}_{m,j}(\nabla)\big\},
\end{equation*}
for $\nabla\in \text{Diff}_{m}(\mathscr{A})_j.$
\begin{lem}
$H_{\mathfrak{s}_{k,i}(\Delta)}$ is a graded derivation of the algebra of graded symbols. 
\end{lem}
 We often call $H_{\mathfrak{s}_{k,i}(\Delta)}$ the \emph{Hamiltonian vector field} corresponding to the \emph{Hamiltonian} $\mathfrak{s}_{k,i}(\Delta).$ Moreover, we have a map
$H:\text{Smbl}_{\bullet,*}(\mathscr{A})\rightarrow D\big(\text{Smbl}_{\bullet,*}(\mathscr{A})\big)_{\mathcal{G}},$ which sends a Hamiltonian to the corresponding Hamiltonian vector field $
\mathfrak{s}_{k,i}(\Delta)\mapsto H_{\mathfrak{s}_{k,i}(\Delta)}.$
Such an assignment
defines a derivation of $\text{Smbl}_{\bullet,*}(\mathscr{A})$ with values in the $\text{Smbl}_{\bullet,*}(\mathscr{A})$-module $D_{\mathcal{G}}\big(\text{Smbl}(\mathscr{A}),\text{Smbl}(\mathscr{A})\big).$
The totality of graded Hamiltonian derivations define a graded sub-$\text{Smbl}(\mathscr{A})_{\bullet}$- module of $D\big(\text{Smbl}(\mathscr{A})\big),$ often denoted by
$\text{Ham}_{\mathcal{G}}(\mathscr{A}).$
\subsubsection{Biderivations, multiderivations and additional functors of DCGCA}
In addition to first order graded derivations, we will need to understand multiderivations as well, paying special attention to bi-derivations. To this end, let us introduce and give precise meaning to a functor,
\begin{eqnarray}
\label{eqn:BiDer}
D_2^{\mathcal{G}}:\mathrm{gMod}^{\mathcal{G}}(\mathscr{A})&\rightarrow&\mathrm{Mod}^{\mathcal{G}}(\mathscr{A}) \nonumber
\\
\mathscr{P}&\mapsto& D_2^{\mathcal{G}}(\mathscr{P}):=D\big(\mathscr{A},D(\mathscr{P})\subset \text{Diff}_{1}^>(\mathscr{P})_{\mathcal{G}}\big),
\end{eqnarray}
where $\text{Diff}_1^>$ is the functor of first order differential operators supplied with its right module structure. Elements of this module are viewed as graded derivations $\Delta:\mathscr{A}\rightarrow\text{Diff}_1^>(\mathscr{A}),$ such that $\Delta(\mathscr{A})\subset D(\mathscr{A}).$ This functor co-represents the module of two-forms in the graded algebraic setting. Indeed, this functorial algebraic definition coincides with the usual one in the following way. Setting $\mathscr{P}=\mathscr{A}$ for simplicity, given a map $\Delta:\mathscr{A}\times\mathscr{A}\rightarrow\mathscr{A}$ such that
$\Delta(a,-)$ is a derivation of $\mathscr{A}$, one has the graded Leibniz rule,
$$\Delta(ab,-)=(-1)^{\Delta b+ab}b^{>}\Delta(a,-)+(-1))^{a\Delta}a^{>}\Delta(b,-),$$
for $a,b\in \mathscr{A},$ where $a^>(-)$ denotes the right module action.

Now for every $\Delta \in D_2^{\mathcal{G}}(\mathscr{P})$ there is a naturally associated mapping, $
\tilde{\Delta}:\mathscr{A}\times\mathscr{A}\rightarrow\mathscr{P},\hspace{2mm}
(a,b)\longmapsto \tilde{\Delta}(a,b):=\big[\Delta(a)\big](b),$
for all homogeneous $a,b\in\mathscr{A}.$

\begin{lem}
\label{biderdual}
The map $\tilde{\Delta}$ is a graded bi-derivation which is graded-skew symmetric. That is we have
\begin{enumerate}
\item  (Derivation in first entry) 
$\tilde{\Delta}(a_1\cdot a_2,b)=(-1)^{a_1\cdot \Delta}a_1\cdot\tilde{\Delta}(a_2,b)+(-1)^{a_2\cdot(\Delta+a_1)}\tilde{\Delta}(a_1,b)\cdot a_2,$

\item  (Derivation in second entry)
$\tilde{\Delta}(a,b_1\cdot b_2)=(-1)^{b_1\cdot(\Delta+a)}b_1\cdot\tilde{\Delta}(a,b_2)+(-1)^{b_2\cdot (b_1+\Delta+a)}b_2\tilde{\Delta}(a,b_1),$

\item  (Graded skew-symmetricity)
$\tilde{\Delta}(a,b)=-(-1)^{a\cdot b}\tilde{\Delta}(b,a).$
\end{enumerate}
\end{lem}

Using the functors defined above, one may immediately obtain various functors of graded multi-derivations, and multi-differential operators in an inductive manner.
For example, one may consider
$\mathfrak{D}_2(\mathscr{P}):=\text{Diff}_1^>\big(\mathscr{A},D(\mathscr{P})\subset \text{Diff}_1^>(\mathscr{P})\big)$ and notice that we have an embedding of graded $\mathbb{K}$-modules $D_2(\mathscr{P})\subset \mathfrak{D}_2(\mathscr{P}).$ This embedding is first order graded differential operator of degree zero, so it is meaningful to define the functors
$$D_k(\mathscr{P}):=D\big(\mathscr{A},D_{k-1}(\mathscr{P})\subset \mathfrak{D}_{k-1}^>(\mathscr{P})\big)$$ for each $k\geq 1$ where 
$\mathfrak{D}_k(\mathscr{P}):=\text{Diff}_1^>\big(\mathscr{A},D_{k-1}(\mathscr{P})\subset \mathfrak{D}_{k-1}(\mathscr{P})\big).$
\begin{rmk}
When $\mathscr{A}=A$ is an ungraded algebra, with $A=C^{\infty}(M)$, then $D_k(A)$ consists of $k$-vector fields on $M$.
\end{rmk}

In particular, we may compute the degree $3$  multiderivations from our knowledge of first and second degree derivations,  $D_3(-):=D\big(D_2(-)\subset \mathfrak{D}_2^>(-)\big).$
We refer the reader to \cite{Vin01},\cite{KraVerb}, for details, with the obvious generalizations to the graded setting.
We remark that one aspect of formalizing differential calculus in this manner is that we can generate many other new functors of differential calculus just by substituting in place of $D_k$ above, the functor $\text{Dif}_k(\mathscr{P}):=\{\Box\in\text{Diff}_k|\Box(1)=0\},$ in the definition of the graded multiderivations.
In this way, one arrives at various higher degree analogues of multi-derivation functors and their naturally associated complexes, for instance Diff-Spencer complexes or the higher order de Rham complexes of \cite{VezVin}.

 \subsubsection{The $\mathcal{G}$-graded Schouten bracket}
The totality of multiderivations of a graded algebra $D_*(\mathscr{A})_{\mathcal{G}}$ naturally carries a $\mathcal{G}\oplus \mathbb{Z}$ graded commutative algebra structure with respect to the (graded) wedge product $\Delta_g\wedge \nabla_h=(-1)^{g\cdot h+ij}\nabla_h\wedge \Delta_g$ for $\Delta_g\in D_i(\mathscr{A})_g,\nabla_h\in D_j(\mathscr{A})_h.$ This is defined by induction on $i+j$ as follows.
For $i+j=0,$ the quantity $\Delta_g\wedge\nabla_h$ is simply multiplication in $\mathscr{A}.$ For $i+j>0,$ the object $\Delta_g\wedge \nabla_h$ satisfies
$$(\Delta_g\wedge\nabla_h)(a)=(-1)^{a\nabla+j}\Delta_g(a)\wedge\nabla+\Delta_g\wedge\nabla_h(a).$$
Furthermore, it is well known that the graded algebra of multi-derivations has a natural bracket $-$ the \emph{Schouten bracket}, which is extended from the Lie algebra on $D_1(\mathscr{A})_{\mathcal{G}}$ as a graded derivation with respect to this wedge product, and endows $D_*(\mathscr{A})_{\mathcal{G}}$ with the structure of a graded Lie algebra.
This operation
\begin{equation}
\label{eqn:schouten}
[\![-,-]\!]:D_i(\mathscr{A})\otimes D_j(\mathscr{A})\rightarrow D_{i+j-1}(\mathscr{A}),
\end{equation}
is given on homogeneous components of degrees $g,h\in\mathcal{G}$ by
$[\![D_{i}(\mathscr{A})_g,D_{j}(\mathscr{A})_h]\!]\subseteq  D_{i+j-1}(\mathscr{A})_{g+h}.$
This bracket is inductively defined on the order $(i+j),$ by setting
$[\![X,a]\!]:=X(a),$ for $X\in D_1(\mathscr{A})$ and $a\in\mathscr{A}$ and for $i+j>1,$ we set
$
[\![a,\nabla_h]\!]:= (-1)^{a\cdot h +j}[\![\nabla_h,a]\!]=(-1)^{a\cdot h +j}\nabla_h(a),$
and
$$
[\![\Delta_g,\nabla_h]\!](a):= [\![\Delta_g,\nabla_h(a)]\!]-(-1)^{a\cdot h+j}[\![\Delta_g(a),\nabla_h]\!] ,(i+j)>0,
$$
for homogeneous $a\in\mathscr{A}, \Delta_g\in D_{i}(\mathscr{A})_g$ and $\nabla_h\in D_{j}(\mathscr{A})_h$.

\begin{lem}
The operation $[\![-,-]\!]$ is well-defined and on $D_1(\mathscr{A})$ coincides with the graded commutator. This bracket enjoys the properties,
\begin{enumerate}
    \item $[\![\Delta_g,\nabla_h]\!]+(-1)^{g\cdot h +
(i+1)(j+1)}[\![\nabla_h,\Delta_g]\!]=0,$
\item $[\![\Delta_g,\nabla_h \wedge \nabla_f]\!] = [\![\Delta_g,\nabla_h]\!]\wedge \nabla_f+(-1)^{g\cdot h+
(i+1)j}\nabla_h\wedge [\![\Delta_g,\nabla_h]\!],$
\item $
[\![\Delta_g,[\![\nabla_h,\Box_f]\!]]\!]=[\![[\![\Delta_g,\nabla_h]\!],\Box_f]\!]
+(-1)^{\Delta\nabla+(i+1)(j+1)}
 [\![\nabla_h,[\![\Delta_g,\Box_f]\!]]\!],$
\end{enumerate}
for $\Delta_g\in D_{i}(\mathscr{A})_g,\nabla_h\in D_j(\mathscr{A})_h,\Box_f\in D_k(\mathscr{A})_f.$
\end{lem}

One may establish this result by straightforward induction on the order of the operators. Furthermore, we have some additional properties which assist with computations.
\begin{lem}
For $\Delta_g\in D_i(\mathscr{A})_g,$ and $\nabla_h\in D_j(\mathscr{A})_h$, for all $a\in\mathscr{A}$ we have
\begin{enumerate}
    \item  $[\![\Delta_g,a\nabla_h]\!]=\Delta_g(a) \nabla+(-1)^{a\Delta}a[\![\Delta_g,\nabla_h]\!],$
    
    \item  $[\![a\Delta_g,\nabla_h]\!]=a[\![\Delta_g,\nabla_h]\!]+(-1)^{\Delta\nabla+i(j+1)}[\![a,\nabla_h]\!]\Delta_g,$

    \item  $[\![\Delta_g,\nabla_h\wedge a]\!]=[\![\Delta_g,\nabla_h]\!]a+(-1)^{\Delta\nabla+(i+1)j}\nabla_h\wedge\Delta_g(a),$
    
    \item  $[\![\Delta_g\wedge a,\nabla_h]\!]=\Delta_g\wedge [\![a,\nabla_h]\!]+(-1)^{a\nabla}[\![\Delta_g,\nabla_h]\!]\wedge a.$
\end{enumerate}

\end{lem}
Following \cite{KraVin}, we give a graded version of the general notion of a Poisson structure in the purely algebraic setting, where they are renamed as canonical structures.

\begin{definition}
A graded biderivation $\Pi$ which satisfies $[\![\Pi,\Pi]\!]=0,$ is called a \emph{(graded) canonical structure}. The pair $(\mathscr{A},\Pi)$ is a \emph{(graded) canonical algebra.}
\end{definition}

\begin{rmk}
Canonical algebras provide a purely algebraic generalization of smooth function algebras on manifolds with symplectic structure (see \cite{KraVin} for an explanation). In particular, pairs $(\mathscr{A},\Pi)$ provide a totally algebraic version of graded symplectic manifolds.
\end{rmk}

\subsubsection{The $\mathcal{G}$-graded Schouten-Jacobi bracket}
In analogy with the discussion of bi-derivations and higher degree multi-derivations given above, there is a canonical bracket on the space of multi-differential operators. 
To see this, let the functor of graded bi-differential operators, of order $1$ in each entry be defined for $\mathscr{A}$ any $\mathcal{G}$-graded algebra as the assignment
$$\text{Diff}_1\big(\mathscr{A},\text{Diff}_1^>(\mathscr{A},-)_{\mathcal{G}}\big)_{\mathcal{G}}:\text{gMod}^{\mathcal{G}}(\mathscr{A})\rightarrow \text{gMod}^{\mathcal{G}}(\mathscr{A}),$$
and denote such bi-differential operators as
\begin{equation}
\label{eqn:BiDiffops}
\text{Diff}_1^{(2)}(\mathscr{A},\mathscr{A})_{\mathcal{G}}:=\text{Diff}_1\big(\mathscr{A},\text{Diff}_1(\mathscr{A},\mathscr{A})_{\mathcal{G}}\big)_{\mathcal{G}}.
\end{equation}
These are skew-symmetric first order graded differential operators in each entry and this definition may be extended to higher orders in a straightforward way.
\begin{lem}
\label{equivbidiffop}
Any $\Box\in \mathrm{Diff}_1^{(2)}(\mathscr{A},\mathscr{A})_{\mathcal{G}}$ can be identified with a bilinear graded map $\tilde{\Box}:\mathscr{A}\times\mathscr{A}\rightarrow \mathscr{A},$ which is a graded differential operator of order $1$ in each entry.
\end{lem}
Explicitly, suppose $\Box_g\in\text{Diff}_1^{(2)}(\mathscr{A},\mathscr{A})_g,$ is of homogeneous degree $g$. Then $\Box_g^{A_h}:=\Box_g|_{A_h}:A_h\subset \mathscr{A}\rightarrow \text{Diff}_1(\mathscr{A},\mathscr{A})_{\mathcal{G}},$
is related to $\tilde{\Box}_g^{A_h,-}:A_h\times \mathscr{A}\rightarrow \mathscr{A},$ via the formula
$\tilde{\Box}_g(a,b)=\big[\Box_g(a)\big](b),$
for $a\in A_h,$ and $b\in\mathscr{A}.$

The map $\tilde{\Box}$ can be considered as a bilinear bracket operation on $\mathscr{A}$ by
$\{-,-\}:=\tilde{\Box}.$
The bracket is skew-symmetric and defines on $\mathscr{A}$ the structure of a $\mathbb{K}$-pre-Lie algebra (or left symmetric algebra), which means that for any triple $a_1,a_2,a_3\in\mathscr{A}$, we have
$a_1(a_2a_3)-a_2(a_1a_3)=(a_1a_2)a_3-(a_2a_1)a_3.$
In fact, the Jacobiator
\begin{eqnarray*}
\text{Jac}(\tilde{\Box})(a_1,a_2,a_3)&:=&\tilde{\Box}\big(\tilde{\Box}(a_1,a_2),a_3\big)- \tilde{\Box}\big(a_1,\tilde{\Box}(a_2,a_3)\big)-(-1)^{(a_1+g)(a_2+g)}\tilde{\Box}\big(a_2,\tilde{\Box}(a_1,a_3)\big)
\\
&\equiv&\big\{\{a_1,a_2\},a_3\big\} +\text{cyclic permutations} =0,
\end{eqnarray*}
and the algebra $\mathscr{A}$ is a $\mathbb{K}$-Lie algebra. 
The bracket $\{-,-\}$ then coincides with the canonical
Jacobi-Schouten bracket on the space of multi-differential operators. If $\tilde{\Box}$ has degree $g$, then this structure is defined by the ``master equation'' 
\begin{equation}
    \label{eqn:SchJac}
[\![\Box_g,\Box_g]\!]^{SJ}=2\text{Jac}(\tilde{\Box}_g)=0,
\end{equation}
Extending the definition \ref{eqn:BiDiffops} to higher orders, we see that $\text{Diff}_1^{(i)}(\mathscr{A},\mathscr{A})_g$ are those $g$-graded morphism $\mathscr{A}\rightarrow \text{Diff}_1^{(i-1)}(\mathscr{A},\mathscr{A}),$ such that $\Delta$ is a first order poly-differential operator. In fact, this means that $\Delta$ satisfies
$\Delta_g(ab)=(-1)^{ag}a\Delta_g(b)+\Delta_g(a)b-\Delta_g(1)ab.$
Moreover, it satisfies,
$$\tilde{\Delta}_g(a,b)=-(-1)^{(a+g)(b+g)}\tilde{\Delta}_g(b,a).$$
In this case lemma \ref{equivbidiffop} tells us that we have $\tilde{\Delta}_g(a_1,...,a_j)=\Delta_g(a_1,...,a_{j-1})(a_j).$
The graded product in $\mathscr{A}$ can be exteneded to $\text{Diff}_1^{(\bullet)}(\mathscr{A},\mathscr{A}):=\bigoplus_{i=0}^{\infty}\text{Diff}_{1}^{(i)}(\mathscr{A},\mathscr{A})_{\mathcal{G}},$ by setting
$\Box_h\wedge \Delta_g(a):=\Box_h\wedge \Delta_g(a)+(-1)^{ag+j}\Box_h(a)\wedge \Delta_g,$
for all $a\in\mathscr{A}$ and for $\Delta_g\in\text{Diff}_1^{(j)}(\mathscr{A},\mathscr{A})_g$ and for $\Box_h$ an element of arbitrary weight, and degree $h$.
This satisfies
$\Box_h\wedge \Delta_g=(-1)^{hg + ij}\Delta_g\wedge \Box_h,$ which equivalently reads
$\tilde{(\Box_h\wedge\Delta_g)}(a,b)=(-1)^{(a+g)h}\Delta_g(a)\Box_h(b)-(-1)^{ag}\Box_h(a)\Delta_g(b).$

Analogously to the case of the $\mathcal{G}$-graded Schouten-bracket, we obtain
an inductive definition of a new bracket denoted $[\![-,-]\!]^{SJ}$ which is called the \emph{Schouten-Jacobi} bracket. It is defined by setting $[\![a,b]\!]^{SJ}=0,$ for all $a,b\in\mathscr{A},$ viewed as order zero weight zero graded multi-differential operators.
Then we set
$[\![\Delta_g,a]\!]^{SJ}:=\Delta_g(a),$ as well as
$[\![a,\Delta_g]\!]^{SJ}=-(-1)^{ag+i}\Delta_g(a), a\in\mathscr{A}, \Delta_g\in\text{Diff}_{1}^{(i)}(\mathscr{A},\mathscr{A})_g.$ Finally, one sets 
$$[\![\Delta_g,\Box_h]\!]^{SJ}(a):=[\![\Delta_g,\Box_h(a)]\!]^{SJ}+(-1)^{ah+j-1}[\![\Delta_g(a),\Box_h]\!]^{SJ},\hspace{1mm} i+j>0.$$
Such a bracket defines a $\mathcal{G}$-graded Lie algebra structure on the $\mathscr{A}$-module $\text{Diff}_{1}^{(\bullet)}(\mathscr{A},\mathscr{A})_{\mathcal{G}},$ with the skew-symmetricity of this bracket with respect to the $\mathcal{G}\oplus \mathbb{Z}$-structure is 
$[\![\Delta_g,\Box_h]\!]^{SJ}=-(-1)^{gh+(i-1)(j-1)}[\![\Box_h,\Delta_g]\!]^{SJ},$
where $\Delta_g\in\text{Diff}_{1}^{(i)}(\mathscr{A},\mathscr{A})_g$ while $\Box_h\in\text{Diff}_{1}^{(j)}(\mathscr{A},\mathscr{A})_h.$
However, since
$$[\![\Delta_g,ab]\!]^{SJ}=\Delta(ab)=\Delta_g(a)b-(-1)^{ag}a\Delta_g(b)-\Delta_g(1)ab=[\![\Delta_g,a]\!]^{SJ}b+(-1)^{ag}a[\![\Delta_g,b]\!]^{SJ}-[\![\Delta_g,1]\!]^{SJ}(ab),$$
instead of having a graded Leibniz rule, one has 
$$[\![\Delta_g,\Box_{h_1}\wedge\Box_{h_2}]\!]^{SJ}=[\![\Delta_g,\Box_{h_1}]\!]^{SJ}\wedge \Box_{h_2}+(-1)^{gh_1+(i-1)j}\Box_{h_1}\wedge [\![\Delta_g,\Box_{h_2}]\!]^{SJ}-[\![\Delta_g,1]\!]^{SJ}\wedge \Box_{h_1}\wedge \Box_{h_2},$$
whih is called the ($\mathcal{G}$-graded) \emph{Jacobi-Leibniz rule}.

Consider now $\mathcal{G}$ being either $\mathbb{Z}$ or $\mathbb{N}$ or an appropriate integer power, ie. $\mathbb{Z}^n.$

\begin{definition}
A degree $g$ bi-differential first order operator $\Box_g\in\mathrm{Diff}_{1}^{(2)}(\mathscr{A},\mathscr{A})_g$ is said to be a \emph{(graded) canonical structure of Jacobi type}, in $\mathscr{A}$ if $g\cdot g$ is even and if 
$[\![\Box_g,\Box_g]\!]^{SJ}=0.$
\end{definition}

The above definition coincides upon restriction to the sub-functor $\text{Dif}_{1}^{(2)}\cong D_2,$ with the notion of a canonical structure of Poisson type (ie. Poisson structure). Actually, a Poisson structure is simply a graded-canonical structure of Jacobi type such that $\tilde{\Box}_g(1,-)=0.$ 
\section{Diolic commutative algebra}
\label{sec:dioles}
We now introduce the main objects of study in this paper, first in a general manner and then specifying to our case of interest. Let $(C,\otimes)$ be a category satisfying the assumptions of section \ref{ssec:conventions}.
\begin{definition}
\label{DioleDef}
 A \emph{diolic algebra internal to} $\mathrm{C}$, is an object $\mathscr{A}$ of $\mathrm{cAlg}\big(\mathrm{C}^{\mathbb{Z}}\big)$ such that $\mathscr{A}_i:=0,$ for all $i\neq 0,1$ and $\mathscr{A}_1\cdot \mathscr{A}_1:=0.$
\end{definition}
For the purposes of modelling the differential geometry of vector bundles in an algebraic capacity we present two
variations on the notion of an algebra of dioles that deserve our attention.

\begin{definition} A \emph{diolic} $\mathbb{K}$-\emph{algebra}, is a diolic algebra internal to $\mathrm{Vect}(\mathbb{K}).$
\end{definition}
Explicitly, a diolic $\mathbb{K}$-algebra is a $\mathbb{Z}$-graded algebra
$\mathscr{A}:=
\mathscr{A}_0\oplus \mathscr{A}_1,$
with homogeneous components prescribed as 
$\mathscr{A}_0:=A,\mathscr{A}_1:=P,$
and $\mathscr{A}_i:=\emptyset, \forall i\neq 0,1,$
for $A\in\text{Alg}(\mathbb{K}),P\in\text{Mod}(A),$ with the only nontrivial multiplications given by $A\cdot A\subseteq A,$ as well as the $\mathscr{A}_0\cdot \mathscr{A}_1=A\cdot P\subseteq P,$ the $A$-module multiplication in $P$. Moreover, we have the important property that $P\cdot P\subseteq \mathscr{A}_2=\emptyset,$ so in this sense, diolic algebras are square zero. 
\begin{rmk}
With the exception of sub-section (\ref{sssec:PVA}), we only study diolic algebras internal to $\mathbb{K}$-vector spaces, and will thus call them simply \emph{diolic algebras}, written as $\mathscr{A}=A\oplus P.$ 
\end{rmk}
One can define morphisms between diolic algebras as simply morphisms of graded commutative algebras, however for the purposes of modelling differential geometry, this class is too general. In particular, not all morphisms of diolic algebras yield a meaningful notion when we pass to the differential geometric setting, which is to say, are not realized as vector bundle maps. Moreover, we will focus our attention on diolic algebras which have a fixed degree zero component, constituting a full-subcategory of all diolic algebras.

To this end, consider $\text{Diole}_{\mathbb{K}}(A)$ which is defined as a subcategory of $\mathrm{gcAlg}^{\mathbb{Z}}$ as follows.
An object $\mathscr{A}\in\text{Diole}_{\mathbb{K}}(A)$ is a diole algebra with fixed degree zero component $A$, and where morphisms $f:A\oplus P\rightarrow A\oplus Q,$ are pairs $f:=(id_A,f_1)$ where $id_A$ is the identity algebra morphism and where $f_1:P\rightarrow Q$ is a morphism of $A$-modules.
This category is called the category of \textit{dioles over} $A.$ 
The assignment of such a category to for any commutative $\mathbb{K}$-algebra is a functorial one, since for any  morphism of commutative algebras $\phi:A\rightarrow B$ (to be thought of as arising from the pull-back $\varphi:=f^*:C^{\infty}(M)\rightarrow C^{\infty}(N)$ of a  smooth map of manifolds) we have a functor 
$\mathcal{D}i(\phi):\text{Diole}_{\mathbb{K}}(A)\rightarrow \text{Diole}_{\mathbb{K}}(B),$
which is induced by base change. Specifically, $\phi$ allows us to view $B$ as an $A$-module and consequently any $A$-module $P$ gives rise to $B\otimes_A P,$ which can be endowed with a $B$-module structure via the action $b_1\cdot (b\otimes p):=\mu_B(b_1.b)\otimes p.$
Consequently for any diole over $A$, say we have an induced diole over $B$. 
\begin{rmk}
Defined in this way, the graded differential calculus in the category $\text{Diole}_{\mathbb{K}}(A)$ recovers, among other things, the ungraded differential calculus in the category $\text{Mod}(A).$
\end{rmk}

Given an algebra of dioles over a fixed algebra $A$, we may consider what are modules over it (in the sense of subsection \ref{ssec:conventions}).
\begin{definition}
A \emph{diolic module}, $\mathscr{P}$ is a graded (left) $\mathscr{A}$-module.
\end{definition}
Consequently, this means that $\mathscr{P}$ is a $\mathbb{Z}$-graded $\mathbb{K}$-module with an $\mathscr{A}$-action, which reads 
$\mathscr{A}_0\cdot \mathscr{P}_g\subseteq \mathscr{P}_g,$ when restricted to the degree zero component of $\mathscr{A}$ and
$P\cdot \mathscr{P}_g\subseteq \mathscr{P}_{g+1},$
for all $g\in \mathbb{Z}.$ In particular, each component is an $A$-module.
We obtain an obvious category of $\mathscr{A}$-modules, to be denoted $\text{Mod}^{\mathbb{Z}}(\mathscr{A})$. 
By turning our attention to the condition $P\cdot \mathscr{P}_g\subseteq \mathscr{P}_{g+1},$ we may define another category of diolic modules which will be particularly interesting.

\begin{definition}
A \emph{truncated diolic module} is a diolic module $\mathscr{Q}=\mathscr{Q}_0\oplus\mathscr{Q}_1$ with $\mathscr{Q}_i:=\emptyset,\forall i\neq 0,1$ where $\mathscr{Q}_0,\mathscr{Q}_1\in\text{Mod}(A),$ endowed with an $A$-module homomorphism (degree zero), 
$\varphi:\mathscr{A}_1\otimes_A\mathscr{Q}_0\rightarrow \mathscr{Q}_1,
$
called the \emph{diolic structure homorphism} or simply the \emph{structure map.}
\end{definition}
We write $(\mathscr{Q},\varphi)$ for a truncated diolic module, or simply $\mathscr{Q}$ when the structure map is understood.
\begin{exm}
The pair $\big(\mathscr{A},id_P\big)$ is a truncated diolic module. Explicitly, $\mathscr{A}$ is a truncated diolic module over itself, with the trivial diolic structure map,
$id_P:P\cong P\otimes_A A\rightarrow P.$
\end{exm}

\begin{exm}
\label{wedgediole}
Consider $\Lambda^1(A)\oplus \Lambda^1(P)$ with the `wedge' morphism 
$\wedge:\Lambda^1(A)\otimes_A P\rightarrow \Lambda^1(P), \omega\otimes p\longmapsto \omega\wedge p.$
This is a truncated diolic module.
\end{exm}

\begin{exm}
Let $\phi:A\rightarrow B$ be a morphism of commutative $\mathbb{K}$-algebras. Let $\overline{\phi}:P\rightarrow Q$ be an $A$-module morphism between an $A$-module $P$ and a $B$-module $Q$ considered as an $A$-module via $\phi,$ denoted $Q^{<\phi}.$ Consider the diolic algebra over $A$ determined by $P$, $\mathscr{A}=A\oplus P.$ Then there is a truncated diolic module associated to the morphism $\overline{\phi}$ given by
$\mathscr{P}_{\overline{\phi}}:=B\oplus Q^{<\phi},$
with diolic structure map
$\beta_{\overline{\phi}}:B\otimes_A P\rightarrow Q^{<\phi},$ given by the universal property of scalar extensions. 
\end{exm}
The above example provides us with a conceptual environment for studying differential calculus along maps, so for instance, along sub-manifolds.

To complete the definition of the category of truncated diolic modules, we define what are morphisms. The appropriate notion is inherited from the usual notion of a morphism of graded $\mathscr{A}$-modules.
\begin{definition}
Let $(\mathscr{Q},\beta),(\mathscr{R},\gamma)$ be two truncated diolic modules over $\mathscr{A}.$ A \emph{morphism of truncated diolic modules}, $\Phi:(\mathscr{Q},\beta)\rightarrow(\mathscr{R},\gamma)$ is a pair
$\Phi=(\varphi_0,\varphi_1)$ where $\varphi_0:Q_0\rightarrow R_0$ and $\varphi_1:Q_1\rightarrow R_1$ are $A$-module homomorphisms, which satisfy the compatibility relation
$\varphi_1\circ \beta=\gamma\circ(id_P\otimes \varphi_0).$
\end{definition}
We write the compatibility relation for $p\otimes q_0\in P\otimes_A Q_0$ as
$\phi_1\big(\beta(p\otimes q_0)\big)=\gamma\big(p\otimes \varphi_0(q_0)\big)$ in $R_1.$
Denote the category of truncated diolic modules as $\tau\text{Mod}(\mathscr{A}),$ 
and call such an object $(\mathscr{Q},\varphi)$ \emph{projective} if each component $Q_i$ is a projective $A$-module.
Furthermore, we say it is \emph{non-degenerate} if the associated adjoint homomorphism to the structure map $\varphi^{\sharp}:P\rightarrow \text{Hom}_A( Q_0,Q_1),$ is an isomorphism of $A$-modules.

Diolic-like algebraic structures can be considered in a wider context than provided above and provide a rather self-contained environment for studying certain mathematical structures. Two prominent examples are Lie algebra objects in the category of diolic $\mathbb{K}$-algebras and
and Poisson algebra objects in the category of diolic $\mathbb{K}[\partial]$-algebras.

\subsubsection{Diolic Lie algebras}\label{sssec:DLAlg}
One can pass to a more general notion of being \emph{diolic} as a mathematical object - that is, graded and suitably truncated, encoding a square-zero property. Using this idea, we may ask what a Lie algebra object in the category of diolic algebras looks like.
\begin{definition}
A \emph{truncated diolic Lie algebra} is a graded $\mathbb{K}$-vector space $\mathfrak{g}=\mathfrak{g}_0\oplus \mathfrak{g}_1$ with a Lie bracket $[-,-]$ (a skew-symmetric operation satisfying the Jacobi identity) whose homogeneous components are defined by the operations $[\mathfrak{g}_0,\mathfrak{g}_0]\subseteq \mathfrak{g}_0,$ and $[\mathfrak{g}_0,\mathfrak{g}_1]\subseteq \mathfrak{g}_1,$ such that $\mathfrak{g}_1\cdot \mathfrak{g}_1=0.$
\end{definition}
Studying this object provides a concise formalism for studying Lie algebras for the following reason.
Consider the Jacobi identity for triples $a,b\in\mathfrak{g}_0$ and $p\in\mathfrak{g}_1.$
It looks as
$$\big[[a,b]_0,p\big]_1+\big[[b,p]_1,a\big]_1+\big[[p,a]_1,b\big]_1=0,$$
where we write $[-,-]_i$ for the situation when the sum of the degrees of the entries is $i.$ 
By rearranging this expression, we arrive at a more convenient format as
$\big[[a,b]_0,p\big]_1=\big[a,[b,p]_1\big]_1-\big[b,[a,p]_1\big]_1.$
Define for any $a\in\mathfrak{g}_0$ the morphisms
$\text{ad}_a:\mathfrak{g}_1\rightarrow\mathfrak{g}_1,\text{ad}_a(p):=[a,p]_1,$
which is an endomorphism of $\mathfrak{g}_1.$
The above relation becomes
$\text{ad}_{[a,b]_0}(p)=\big[\text{ad}_a,\text{ad}_b\big](p),$

Consequently any diolic Lie algebra contains information of a representation of $\mathfrak{g}_0$ in $\mathfrak{g}_1.$ Conversely, any representation of a Lie algebra determines a diolic Lie algebra. 
\begin{lem}
Representations of $\mathfrak{g}_0$ in $\mathfrak{g}_1$ and diolic Lie algebra structures over $\mathfrak{g}_0$ are in one-to-one correspondence.
\end{lem}

As is possible with any finite dimensional Lie algebra with a representation, we may study the corresponging Lie algebra cohomology.
In our graded framework, in the category of diolic Lie algebras, the degree $1$ Chevalley-Eilenberg complex for $\mathfrak{g}$ interpreted as a Lie algebra representation of $\mathfrak{g}_0$ on the finite dimensional vector space $\mathfrak{g}_1$ coincides with the ordinary Lie algebra cohomology. Namely for any representation
$\rho:\mathfrak{g}_0\rightarrow \text{End}(\mathfrak{g}_1)$ of a Lie algebra $\mathfrak{g}_0$ on a vector space $\mathfrak{g}_1$ we have the standard Lie algebra
cochain complex $C^{\bullet}(\mathfrak{g}_0,\mathfrak{g}_1)$ of $\mathfrak{g}_0$ with coefficients in $\mathfrak{g}_1$, where $C^p(\mathfrak{g}_0,\mathfrak{g}_1):=\text{Hom}\big(\bigwedge^p\mathfrak{g}_0,\mathfrak{g}_1)$ 
 and for $\tau$ such a $p$-cochain we have a differential $d$ defined in terms of the representation $\rho$ as
 \begin{eqnarray*}
 d^p\tau (a_0,...,a_{p+1})&=&\sum_{i=1}^{p+1}(-1)^{i-1}\rho(a_i)\tau\big(a_0,..,\hat{a}_i,..,a_{p+1}\big)
 \\
 &+&\sum_{i<j}(-1)^{i+j}\tau\big([a_i,a_j],a_1,..,\hat{a}_i,..,\hat{a}_j,..,a_{p+1}\big).
 \end{eqnarray*}
This is precisely the degree $1$ Chevalley-Eilenberg cochains. Later we will see that this degree $1$ Lie algeba cohomology which represents the Lie algebra cohomology of the diolic derivations, is dual to the \textit{Der-complex} of \cite{Rub01}.
In this way, we see that the graded  differential calculus over diolic algebras provides a new setting for studying commonplace mathematical items such as Lie algebra cohomologies.

\subsubsection{Diolic Poisson vertex algebras}\label{sssec:PVA}

Understanding a Poisson $\mathbb{K}$-algebra as a commutative unital algebra which is also a Lie $\mathbb{K}$-algebra, whose bracket satisfies the Leibniz rule as a derivation of $A$, one may consider also Poisson vertex algebras. These algebras constitute important objects of study and play a central role in conformal field theory (CFT) and include Lie-conformal algebras and the theory of $\lambda$-brackets, which describe the singular part of the operator product expansion in $2d$-CFT. 
These are Poisson algebra and Lie algebra objects in the category of $\mathbb{K}[\partial]$-modules, with $\partial$ some endomorphism.
Let $\mathcal{P}$ be a graded Poisson-vertex algebra with a commutative unital algebra structure. This means we have $\partial(a,b)=\partial a\cdot b+a\cdot \partial b$ and the structure of a Lie conformal algebra $\{-\lambda-\}$ which acts by derivations.
To study diolic PVA's, this amounts exactly to studying a PVA in degrees $0,1$ whose degree $1$ component is square zero.
The following is straightforwardly established by comparing definitions.

\begin{lem}
Let $\mathcal{P}$ be a diolic Poisson vertex algebra. This datum is equivalent to a Courant-Dorfman algebra structure on its components.
\end{lem}

\subsubsection{`Artificial' diolic algebras}
Suppose that we wish to generate new diole algebras from existing ones in our category $\text{Diole}_{\mathbb{K}}(A).$ 
Defined as a sub-category of $\text{gcAlg}(\mathbb{K})$, the category of dioles over $A$ inherits various natural operations, for instance a monoidal structure. One problem that arises is that such a product \emph{does not} give us the diole associated to a tensor product of modules. That is, if $\otimes^D$ is the tensor product in the category of diolic algebras over $A,$ we have that $(A\oplus P)\otimes^D (A\oplus Q)\neq A\oplus (P\otimes Q).$
Nevertheless, we may take the tensor product of $A$-modules $P\otimes Q$ and simply `build'  a diolic algebra $\mathscr{A}_{\otimes}:=A\oplus (P\otimes Q).$ In this sense, this asssociated diole to the tensor product functor on the category of $A$-modules is said to be \emph{artificial}.
This is a general situation. Let $\mathcal{F}:\text{Mod}(A)\times..\times \text{Mod}(A)\rightarrow \text{Mod}(A)$ be some functor, which may take in finitely many entries, that is representable in the sub-category of vector bundles. 
Some common examples which arise are as follows.

\begin{exm}
\label{examples}
Let $\pi:E\rightarrow M$ and $\pi':E'\rightarrow M$ be two vector bundles with modules of sections $P,Q$, respectively. Then we may form the direct sum bundle, $\pi\oplus \pi':E\oplus E'\rightarrow M$ or the (two-fold) tensor product of two bundles $E\otimes E'.$ More generally take one may take $\bigotimes_{i=1}^nE_i,$ the Hom-bundle with $\text{Hom}(E,E')\xrightarrow{\text{Hom}(\pi,\pi')}M,$ the dual bundle. $E^{\vee}\xrightarrow{\pi^{\vee}}M,$
the $k$'th alernating bundle $\bigwedge^{(k)}E\xrightarrow{\wedge^k\pi}M,$ or the $k$'th symmetric bundle. $\bigodot^kE\xrightarrow{\odot^k\pi}M.$
\end{exm}

With these examples in mind, we pose the definition.

\begin{definition}
Let $\mathcal{F}\in\mathrm{Fun}\big(\prod^n\mathrm{Mod}(A),\mathrm{Mod}(A)\big)$ be an arbitrary functor on the category of vector bundles (or more generally, arbitrary $A$-modules).
The \emph{artificial diole associated to} $\mathcal{F}$, or simply the \emph{artificial} $\mathcal{F}$-\emph{diole}, is the diolic algebra over $A,$ defined as
$\mathscr{A}_{\mathcal{F}}:=A\oplus \mathcal{F}(P_1,P_2,..,P_n),$
for some $A$-modules $P_1,.,P_n.$
\end{definition}
This definition makes possible the ability to generate new algebras of dioles that are not readily obtainable from natural operations in the category $\mathrm{Diole}(A).$ So while it does not solve our tensor-product problem, it provides us with a means by which we may study the tensor product of modules in the setting of diole algebras.
Building similar artificial diolic algebras from suitable functors on the category of $A$-modules is a rather practical way of generating examples of various types of operators arising in differential geometry. Indeed, see example \ref{oplusdiole} below.

\section{Differential calculus in diolic algebras}
\label{sec:diolecalc}
We now undertake the task of computing the various functors of differential calculus given in section \ref{sec:Calc} in the category of base fixed diolic algebras.

\subsection{Diolic derivations}
Let us fix the notation. We denote an operator $\Delta$ of graded degree $g$, as $\Delta_g$, and its restriction to $\mathscr{A}_0=A,\mathscr{A}_1=P$ of the diolic algebra $\mathscr{A}$ will be denoted $\Delta_g^A:=\Delta_g|_{A},\Delta_g^P:=\Delta_g|_{P},$ and we might even omit the reference to the underlying grading $g$ when it is unambiguous. 
\begin{thm}
Let $D_{\mathbb{Z}}(\mathscr{A})$ be the module of ($\mathbb{Z}$)-graded derivations of a diolic algebra $\mathscr{A}=A\oplus P$. Assume that $\text{rank}(P)\neq 1.$ Then elements $X_0$ of
$
D(\mathscr{A})_0$ consist of pairs $X_0=(X_0^A,X_0^P)$ where $X_0^A\in D(A)$ is a derivation $X_0^P:P\rightarrow P$ is a $\mathbb{R}$-linear operator such that $X_0(ap)=X_0^A(a)p+a X_0^P(p),$ holds. Degree $1$ diolic derivations are $P$-valued derivations of $A,$ that is, we have an $A$-module isomorphism $D_1(\mathscr{A})\cong D(A,P).$ 
All other graded derivations are trivial. 
\end{thm}
Those diolic derivations of degree zero are familiar objects - they are \emph{derivations in vector bundles}. Their geometric and local meaning is given below.

\paragraph{The Atiyah sequence}
In \cite{Ati} a robust piece of mathematical machinery was introduced to understand the existence of complex analytic connections in fiber bundles, commonly referred to as \emph{Atiyah sequences}. These sequences appear frequently in the study of the structure of the infinitesimal objects that correspond to Lie groupoids. In particular, they provide an example of a transitive Lie algebroid. 
Our description of diolic derivations indicates that locally there is the following sequence for projective diolic algebras
\begin{equation}
\label{eqn:DiAt1}
0\rightarrow \text{End}(P)\hookrightarrow \text{D}_0(\mathscr{A})\xrightarrow{\sigma} D(A)\rightarrow 0.
\end{equation}
This is a `diolic' reproduction of the usual Atiyah sequence for $P.$

Indeed, by letting $\text{Der}_m(\pi)$ denote the value of the  module $\text{Der}\big(\Gamma(\pi)\big)$ in a point $m\in M$ of the manifold, which is to say the quotient $\text{Der}\big(\Gamma(\pi)\big)/\mu_m\text{Der}\big(\Gamma(\pi)\big),$ where $\mu_m$ is the maximal ideal of $A$, determined by $m$ (consult \cite{Nes} for details), we see elements of $\text{Der}_m(\pi)$ are  $\mathbb{R}$-linear mappings $\overline{X}_m:\Gamma(\pi)\rightarrow E_m,$
 satisfying $\overline{X}_m(f s)=f(m)\overline{X}_m(s)+X_m(f)(s)(m), f\in A, s\in \Gamma(\pi),$ for each tangent vector $X_m\in T_m M.$
We thus find a short exact sequence
\begin{equation}
    \label{Atiyahsequence}
0\rightarrow \text{End}(E_m)\rightarrow \text{Der}_m(\pi)\rightarrow T_mM\rightarrow 0.
\end{equation}
By 'bundling' all fibres together as $\text{Der}(\pi)=\bigcup_{m\in M}\text{Der}_m(\pi),$ we obtain a vector bundle $\text{Der}(\pi)$ over $M$ with fibre $\text{Der}_m(\pi).$
Moreover each $C^{\infty}(M)$-module in \ref{Atiyahsequence} carries a natural Lie algebra structure, with $\text{End}(E_m)$
an ideal in $\text{Der}_m(\pi).$ Therefore, the Atiyah sequence is also a sequence of Lie algebra homomorphisms.
\begin{rmk}
If $P$ coincides with the module of smooth sections of some line bundle $L$ then every first order differential operator $\Gamma(L)\rightarrow \Gamma(L)$ is a derivation of $L$. In particular, consider the trivial line bundle $\mathbb{R}_M:=M\times\mathbb{R}\rightarrow M$. We have $\Gamma(\mathbb{R}_M)=C^{\infty}(M)$ and
first order differential operators $\Gamma(\mathbb{R}_M)\rightarrow \Gamma(\mathbb{R}_M)$ are equivalently derivations of
$\mathbb{R}_M$, given as $X+f:C^{\infty}(M)\rightarrow C^{\infty}(M)$ where $X$ is a vector field
on $M$ and $f\in C^{\infty}(M)$ is interpreted as an operator (multiplication by $f$). There is a natural direct sum decomposition of $C^{\infty}(M)$ -modules $\text{Der}(\mathbb{R}_M)=D(M)\oplus C^{\infty}(M),$ and the the projection $\text{Der}(\mathbb{R}_M)\rightarrow C^{\infty}(M)$ is given by $\Delta\longmapsto \Delta(1).$ 
\end{rmk}
Using our notion of an artifical diole algebra, we may study linear connections in various natural vector bundles by studying the degree zero diolic Atiayh sequence for such diole algebras.
\begin{exm}
\label{oplusdiole}
The artificial $\oplus$-diole, $\mathscr{A}_{\oplus}:=A\oplus \big(P\oplus P'\big)$ yields as $D(\mathscr{A}_{\oplus})_0$ those operators $\Delta\in\text{Der}(\pi)$, $\Delta'\in\text{Der}(\pi')$ such that they have shared scalar type symbol $X.$ Such operators are defined by $(\Delta\oplus \Delta')(p,p'):=\big(\Delta(p),\Delta'(p')\big).$
\end{exm}
\begin{rmk}
Similar analysis can be done for artifical dioles associated to the examples \ref{examples}. The corresponding diolic derivations are as one expects:
$\Delta\otimes\Delta'(p\otimes p'):=\Delta(p)\otimes p'+p\otimes\Delta'(p'),$ as well as 
$\text{Hom}(\Delta,\Delta')(\varphi))(p)=\Delta'\big(\varphi(p)\big)-\varphi(\Delta(p)\big),$ and similarly $\bigwedge^k\Delta(\wedge_I p_I):=\sum_I p_1\wedge..\Delta(s_i)\wedge..\wedge p_k,$ and $\odot^k\Delta\big(\bigodot_I p_I\big):=\sum_{i\in I}p_1\odot..\odot \Delta(p_i)\odot ..\odot p_k.$
\end{rmk}
Let us now investigate the various algebraic structures which exist naturally on $D(\mathscr{A})_{\mathcal{G}}.$
Unsurprisingly, we obtain a completely analogous situation to the non-graded formalism where $D(A)$ is an $A$-module and a Lie algebra with respect to the commutator of vector fields.

 \subsubsection{Module and other algebraic structures}
The fact that the set of vector fields on a manifold is a module over the ring of smooth functions is necessarily valid still for diolic algebras and this is consistent with our formalism.
\begin{lem}
$D(\mathscr{A})$ is a truncated diolic module.
\end{lem}
\begin{proof}
The only non-trivial situation here is analyzing the $P$-module structure in $\text{Der}(P).$
Let $\Delta_0\in \text{Der}(P)$ be a der-operator, and let $p\in P.$ Recall this means $\Delta_0=(\Delta_0^A,\Delta_0^P),$ with $\sigma(\Delta_0^P)=\Delta_0^A.$ One has a well-defined $P$-module structure as
$$(p\cdot\Delta_0)(a)=\Delta_0^A(a)\cdot p,$$
where $p\cdot\Delta_0\in D_{1+0}(\mathscr{A})=D(P),$
since by definition, $(p\cdot \Delta_0)(ab)=\Delta_0^{A}(ab)\cdot p,$ for $a,b\in A$ and since $\Delta_0^A$ is a derivation, by using the Leibniz rule, we find
$(p\cdot \Delta_0)(ab)=\Delta_0^{A}(ab)\cdot p=\Delta_0^A(a)\cdot p b+a\Delta_0^{A}(b)\cdot p=(p\cdot \Delta_0)(a)b+a(p\cdot \Delta_0)(b).$
\end{proof}

\subsubsection{The case of line bundles}
A consequence of the graded algebraic formalism we are working in is the following.
\begin{lem}
\label{deg-1derivations}
A dioic algebra $\mathscr{A}$ admits a further component $D(\mathscr{A})_{-1}$ of degree $-1$ graded derivations if and only if $P$ is a projective module of rank $1.$ In this case such derivations are simply elements of the $A$-linear dual $P^{\vee}=\mathrm{Hom}_A(P,A).$
\end{lem}
\begin{proof}
Noting that 
$\Delta_{-1}$ acts on homogeneous components with $\Delta_{-1}^P:P\rightarrow A$ only, since by definition of $\mathscr{A}$ we do not have any degree $-1$ component. The graded Leibniz rule on products of the form $a\cdot p$ yield 
$\Delta_{-1}(ap)=a\Delta_{-1}^P(p),a\in A,p\in P.$ So in fact, we have $\Delta_{-1}\in P^{\vee}.$ 
Now, to see that $P$ must have rank one, we compute the quantity $\Delta_{-1}^P(p_1p_2)$ using the graded Leibniz rule and the fact $p_1p_2:=0.$ We get
$0\equiv \Delta_{-1}^P(p_1p_2)=\Delta_{-1}^P(p_1)p_2-p_1\Delta_{-1}^P(p_2),$
which amounts to the relation
$\Delta_{-1}^P(p_1)p_2-p_1\Delta_{-1}^P(p_2)=0.$
Let $u_{\alpha}$ be a basis of $P.$ Write $(\Delta_{-1})_{\alpha}^{\gamma}$ the coefficients determining $\Delta_{-1}.$ Then the above relation amounts to $(\Delta_{-1})_{\alpha}^{\gamma}(u_{\gamma})u_{\beta}-u_{\gamma}\Delta_{\alpha}^{\beta}(u_{\beta})=0,$ which upon contracting indices gives us
$(n-1)(\Delta_{-1})_{\gamma}^{\alpha}=0,$
where $n$ is the rank of $P.$ This shows that unless $n=1,$ (ie. $P$ is a line bundle), there will be no degree $-1$ diolic derivations.
\end{proof}

\begin{rmk}
The above proposition tells us that when studying differential calculus in the category of diolic algebras over $A$, we should pay special consideration to those which coincide geometrically with the situation of having a  line bundle.
\end{rmk}

\subsubsection{Module-valued derivations}
We now want to characterize operators $\Delta:\mathscr{A}\rightarrow \mathscr{R},$ which are graded derivations 
thus generalizing the results obtained above.

\begin{lem}
The module of graded derivations of $\mathscr{A}$, with values in the truncated diolic module $(\mathscr{R},\varphi),$ has degree zero component $D(\mathscr{A},\mathscr{R})_0$ whose elements are pairs $X_0=(X_0^A,X_0^P)$ consisting of $X_0^A\in D(A,R_0)$ and $X_0^P\in D(P,R_1)$ such that $X_0^P(ap)=\varphi(X_0^A(a)\otimes p)+aX_0^P(p)$.
The degree $1$ component coincides with $R_1$-valued derivations of $A,D(\mathscr{A},\mathscr{R})_1\cong D(A,R_1).$
\end{lem}
Those degree zero operators described above, are to be denoted as $\text{Der}_{\varphi}(P,R_1)$ which are called \textit{generalized} $\varphi$-\textit{der-operators} for the reason that, by the graded Leibniz rule, they are characterized by a $\varphi$-der Leibniz rule,
\begin{equation}
\label{eqn:phiderLeib}
X_0^P(ap)=\varphi(X_0^A(a)\otimes p)+aX_0^P(p),
\end{equation}
which shows that their symbols are not merely derivations $X_0^A\in D(A)$ as usual, but are instead $X_0^A\in D(A,R_0),$ twisted by $\varphi.$

We readily obtain a generalization of order one diolic Atiyah sequence, which now now has coefficients in the truncated diolic module $\mathscr{R}$. We call it the order one $\varphi$-diolic Atiyah sequence, which is 
\begin{equation}
\label{eqn:phiDiAt}
0\rightarrow \text{Hom}_A(P,R_1)\hookrightarrow D(\mathscr{A},\mathscr{R})_0\xrightarrow{\sigma_{\varphi}}D(A,R_0),
\end{equation}
where the map
$\sigma_{\varphi}:\text{Der}(\varphi,P,R_1)\rightarrow D(A,R_0),X_0\longmapsto \sigma_{\varphi}(X_0):=X_0^A,$
sends a $\varphi$-der-operator to its generalized symbol.
This map generalizes the ordinary symbol map $\sigma:\text{Der}(P)\rightarrow D(A).$

\begin{exm}
Consider the truncated module $\mathscr{P}_{\Lambda^1}:=(\Lambda^1(A)\oplus \Lambda^1(P),\wedge)$ of example \ref{wedgediole} and the de Rham differential $d:A\rightarrow \Lambda^1(A)$ for $A.$ Suppose $\nabla$ is a linear connection (ie. a splitting of the diolic Atiyah sequence), with covariant differntial $d_{\nabla}.$ Recall this is a map
$d_{\nabla}:P\rightarrow \Lambda^1(P)$ which takes when evaluated on $p\in P$ acts as $d_{\nabla}(p)(X):=\nabla_X(p)\in P.$
Then $d_{\nabla}$ is a degree zero derivation of the diolic algebra $\mathscr{A}$ with values in the truncated diolic module $\mathscr{P}_{\Lambda^1}.$ Namely
$d_{\nabla}\in D\big(\mathscr{A},\mathscr{P}_{\Lambda^1}\big)_0\cong \text{Der}\big(\wedge,P,\Lambda^1(P)\big),$
whose symbol is $\sigma_{\wedge}(d_{\nabla})=d,$ since we have the well known der-Leibniz rule,
$d_{\nabla}(ap)=da\wedge p+a d_{\nabla}(p)\equiv \wedge\big(da\otimes p\big)+ad_{\nabla}(p).$
\end{exm}
The above example shows that the notion of covariant differential, just as the notion of linear connection did, finds a natural conceptual home in the formalism of differential calculus over a diolic algebra. 
\begin{exm}
Consider the truncated diolic module $\mathscr{P}_{\overline{\phi}}$ associated with a morphism of $A$-modules.
One has that $\nabla:\mathscr{A}\rightarrow \mathscr{P}_{\overline{\phi}}$ is a degree zero diolic derivation if and only if it is a pair of operators where $\nabla_0^A\in D(A,B)$ satisfying $\nabla_0^A(ab)=\nabla_0^A(a)\phi(b)+\phi(a)\nabla_0^A(b)\in B,$ meaning it is a derivation along $\phi,$ denoted $D(A)_{\phi},$ and if $\nabla_0^P:P\rightarrow Q^<,$ satisfying 
$$\nabla_0^P(ap)=\beta_{\overline{\phi}}\big(\nabla_0^A(a)\otimes p\big)+a\cdot \nabla_0^P(p)\equiv \nabla_0^A(a)\overline{\phi}(p)+\phi(a)\nabla_0^P(p),$$
since the $A$-module structure in $Q$ is $a\cdot q:=\phi(a)q,$ and the morphism $\beta_{\overline{\phi}}$ is the unique morphism obtained by scalar extensions.
Consequently,
$D(\mathscr{A},\mathscr{P}_{\overline{\phi}}\big)_0$ consists of der-operators along $\overline{\phi}$ over derivations along $\phi.$ If, moreover the $B$-submodule generated by $\text{im}(\overline{\phi})\subset Q$ is faithful, every der-operator along $\overline{\phi}$ is over a unique derivation along $\phi$. We also have that $D(\mathscr{A},\mathscr{P}_{\overline{\phi}}\big)_1$ consists of those operators $\nabla_1^A:A\rightarrow Q^<$ which satisfy the $\phi$-Leibniz rule
$\nabla_1^A(ab)=\nabla_1^A(a)\phi(b)+\phi(a)\nabla_1^A(b),$
for $a,b\in A.$ 
\end{exm}

Finally, we characterize the diolic der-operators. These represent the algebraic description of graded linear connections in the algebraic formalism. To this end, let $(\mathscr{R},\beta)$ be a truncated diolic module. 

\begin{thm}
The $\mathscr{A}$-module of graded Der-operators $\mathrm{Der}(\mathscr{R})$ admits a homogeneous decomposition into degree $0$ and degree $1$ components, with $\nabla_0\in\mathrm{Der}(\mathscr{R})_0$ given by a pair $\nabla_0=(\nabla_0^{R_0},\nabla_0^{R_1})$ of Der-operators, $\nabla_0^{R_0}\in \mathrm{Der}(R_0), \nabla_0^{R_1}\in \mathrm{Der}(R_1),$
whose graded symbol is $\sigma_{\nabla_0}=X_0=(X_0^A,X_0^P)\in D(\mathscr{A})_0,$ with symbols $\sigma_{\nabla_0^{R_0}}=X_0^A=\sigma_{\nabla_0^{R_1}}$ satisfying the rule 
$\nabla_0^{R_1}\beta(p,r_0)=\beta\big(X_0^P(p),r_0\big)+\beta\big(p,\nabla_0^{R_0}(r_0)\big)\}.$
Those in degree $1$, $\nabla_1\in \mathrm{Der}(\mathscr{R})_1$ are operators $\nabla_1^{R_0}:R_0\rightarrow R_1$ satisfying $ \nabla_1^{R_0}(ar_0)=\beta\big(X_1^A(a),r_0\big)+a\nabla_1^{R_0}(r_0).$ That is, degree $1$ diolic Der-operators consist of $R_1$-valued derivations of $R_0$ whose vector-valued symbol $X_1^A\in D(A,P)$ satisfy a generalized $\beta$-Der-Leibniz rule.

\end{thm}

Clearly, these operators differ from the usual Der-operators in fibre bundles, which posses scalar type symbols, while our operators described above have symbols which are vector-valued operators. These types of operators arise naturally in the discussion of differential calculus which preserves inner structures in vector bundles. In particular, they arise naturally as an aspect of triolic differential calculus under investigation in a future work.

\subsection{Coordinates}
Let $\pi:E\rightarrow M$ be a rank $m$ vector bundle over $M$, a manifold of $\dim_{\mathbb{R}}(M)=n.$ Choose a chart $(U,x_1,...,x_n)$ in $M$ and a basis for local sections in $\pi$ as $e_{\alpha}.$
Der-operators are additive $\mathbb{R}$-linear operators $
X_0:P=\Gamma(\pi)\rightarrow P=\Gamma(\pi),$ which act on linear combinations $X_0\big(\sum_{i=1}^mp_ie_i\big)$, whose coefficients $p_i\in A=C^{\infty}(M), i=1,..,m.$ This action is uniquely determined on the basis  $e_{\alpha}$ by
$X_{0}^P(e_k)=\sum_{j=1}^mg_j^ke_k,$
for some functions $g_j^k\in C^{\infty}(M).$ Consequently in matrix notation,
$$\scalemath{0.83}{X_0(p)=\begin{pmatrix}
X &0& ..&0&\\
0 & X& ..& 0&\\
..&..&..&..\\
0& 0&..& X\\
\end{pmatrix}\begin{pmatrix}
p_1\\
p_2\\
.\\
p_m
\end{pmatrix}+\begin{pmatrix}
g_{11} & g_{12} &...& g_{1m}\\
g_{21} & g_{22} &...& g_{2m}\\
.. & ..& ...&..\\
g_{m1} & g_{m2} &...& g_{mm}\\
\end{pmatrix}
\begin{pmatrix}
p_1\\
p_2\\
.\\
p_m
\end{pmatrix}},$$
where $X:=X_{0}^A=\sum_{i=1}^nX_i\partial_i\in D(M),$ a vector field (scalar operator). Let $G=||g_{i,j}||, i,k=1...,m,$ be the matrix of functions and denote by $\mathbb{X}:=\mathbb{I}_{m\times m}X$, the diagonal matrix of scalar first order operators where $\mathbb{I}_{m\times m}$ is the $m\times m$ identity matrix.
Then in a more concise form we write
\begin{equation}
X_0=\mathbb{X}+G.
\end{equation}
This decomposition reflects the fact that locally we have a splitting of the short exact sequence (\ref{eqn:DiAt1}).
We may describe degree one diolic derivations locally in a similar fashion. Let
$X_{1}^A:A\rightarrow P,$ be one such $P$-valued derivation of $A$. It will take values, in the basis $e_{\alpha}$ of $P$, as
$X_1^A:f\longmapsto X_{1}^A(f):=\sum_{\alpha=1}^m\overline{X}^{\alpha}(f)e_{\alpha},$
where $\overline{X}^{\alpha}\in D(A)$ for each $\alpha=1,..,m.$
In matrix notation, in the basis $e_{\alpha}$ of $P$, we have
\begin{equation}
X_1^A(f)=\big(
\overline{X}^1(f), 
\overline{X}^2(f),..,
\overline{X}^m(f)
\big)^T.
\end{equation}

\subsubsection{Commutators}
Using the local description of diolic derivations, one may compute the relevant commutators. 
\begin{lem}
There are only two admissable commutators in the graded Lie algebra of diolic derivations, given by 
\begin{equation*}
\big[\mathbb{X}+G,\mathbb{Y}+H\big]=\big[\mathbb{X},\mathbb{Y}\big]-\mathbb{Y}(G)+\mathbb{X}(H)+\big[G,H\big],
\end{equation*}
for two degree zero derivations and 
\begin{equation}
 \scalemath{0.83}{
[\Box_0,\Delta_1](f)=\begin{pmatrix}
[\Box_{0}^A,Z^1](f)\\
[\Box_{0}^A,Z^2](f)\\
...\\
[\Box_{0}^A,Z^m](f)\\
\end{pmatrix}+\begin{pmatrix}
g_{11}& g_{12}&...& g_{1m}\\
..&.. &...& ..\\
.. & ..& ...&..\\
g_{m1} & g_{m2} &...& g_{mm}\\
\end{pmatrix}\begin{pmatrix}
Z^1(f)\\
Z^2(f)\\
...\\
Z^m(f)\\
\end{pmatrix}},
\end{equation}
for the commutator between a degree one and degree zero diolic derivation. Moreover, the $\mathscr{A}$-module of graded derivations of a diolic algebra, together with the above graded commutator constitute a diolic Lie algebra.
\end{lem}

\subsection{A de Rham-like complex in the Diolic formalism}
We now wish to exhibit that a canonical de Rham-like sequence associated with a vector bundle is naturally encoded in the graded de Rham complex of an algebra of dioles $C^{\infty}(M)\oplus \Gamma(\pi).$
In the algebraic situation, we have a morphism $f:\Lambda^k(A)\times P\rightarrow \{\omega:D(A)\times...\times D(A)\rightarrow P\}=:\Lambda^k(P),$ which is defined by $f(\omega,p)(X_1,..,X_k):=\mu\big(\omega(X_1,..,X_k),p\big),$ for all $X_i\in D(A),p\in P$, where $\mu:A\otimes P\rightarrow P$ is the $A$-module structure on $P.$ In the case when either $D(A)$ or $P$ are projective and finitely generated, we have that, the pair $(\Lambda^k(P),f)$ coincides with the tensor product $\Lambda^k(A)\otimes_A P$ which are called $P$-valued $k$-forms, or \emph{vector-valued} differential forms on $M,$ in the geometric situation.
There are variants of this notion, which we introduce now following the unique terminology of \cite{Fat}.

\begin{definition}
A \emph{semi-fat} differential $k$-form is an alternating $A$-linear morphism 
$\omega\colon\mathrm{Der}(P)\times..\times \mathrm{Der}(P)\rightarrow A,$
and the totality of such morphisms is denoted $\Lambda^k(P;A)$. In the geometric setting we write $\Lambda^k(\pi;M).$
\end{definition}

We have the following relation between semi-fat forms and ordinary differential forms.

\begin{lem}
If $P$ is a faithful $A$-module, then  ordinary forms can be interpreted as semi-fat forms.
\end{lem}
\begin{proof}
Since $P$ is faithful every Der-operator has a unique symbol. Therefore we have an $A$-module homomorphism $\mathrm{Der}(P)\rightarrow D(A)$ sending $\overline{X}$ to the unique derivation $X.$ This
is well-defined and in its turn induces an $A$-module homomorphism
$\Lambda^k(A)\rightarrow \Lambda^k(P;A),$ as follows.
For $\omega\in \Lambda^k(A)$ we can view it as a semi-fat form, $\overline{\omega}$ by defining 
$\overline{\omega}\big(\overline{X}_1,..,\overline{X}_k\big):=\omega(X_1,..,X_k)$ for $\overline{X}_i\in \mathrm{Der}(P)$. 
\end{proof}
\begin{lem}
Suppose $P=\Gamma(\pi)$ is the module of sections of some vector bundle. Then $\Lambda^k(M)\rightarrow \Lambda^k(\pi;M)$ above is a momomorphism of $A$-modules.
\end{lem}

\begin{definition}
\emph{Fat differential} $k$-forms are alternating $A$-linear maps $\overline{\omega}:\mathrm{Der}(P)\times..\times \mathrm{Der}(P)\rightarrow P.$
\end{definition}
We note that the usual definition of the wedge product for differential forms extends immediately to provide a map
$\Lambda^p(P;A)\times \Lambda^q(P;A)\rightarrow \Lambda^{p+q}(P;A),$
as well as a map
$\Lambda^p(P;A)\times \overline{\Lambda}^q(P)\rightarrow \overline{\Lambda}^{p+q}(P).$
They are defined as usual,
$(\omega\wedge \theta)(X_1,..,X_{p+q}):=\sum_{\sigma \in S_{p,q}}(-1)^{|\sigma|}\omega(X_{\sigma}(1),..,X_{\sigma(p)})\cdot \theta(X_{\sigma(p+1)},..,X_{\sigma(p+q)}),$ where each $X_i\in\mathrm{Der}(P).$ In the former case, the product on the right hand side is the $A$-multiplication whereas in the latter case, it uses the $A$-module structure of $P$.
Therefore, $\Lambda(P;A)\in\mathrm{gcAlg}(\mathbb{K}),$ while the fat differential forms are graded modules over semi-fat forms.
In the case where $P$ is faithful, we follow the usual intuition for defining the de Rham differential to obtain the \emph{semi-fat} differential. The reader may find details in \cite{Fat}. The same formal computation as for the usual de Rham differential proves that the square of the semi-fat differential vanishes. This way $\Lambda^{\bullet}(P;A)$ is a complex-the \emph{semi-fat de Rham complex}. 

\subsubsection{Diolic differential forms}
Fix a splitting of the diolic Atyah sequence. Thus we identify $D(\mathscr{A})_0$ with the $A$-module of Der-operators.
Identifying differential $k$-forms of degree $0$ on $\mathscr{A}$ with alternating maps $\bigwedge^k D(\mathscr{A})_0\rightarrow A$, and by applying the (degree zero) de Rham differential to the homogeneous component $A$ which one extends to the entire algebra via the graded Leibniz rule, we deduce the following.
\begin{lem}
By restricting the de Rham differential to $A \subset \mathscr{A}$, we obtain the homogeneous de Rham complex of $\mathscr{A}$ in degree zero coincides with the semi-fat de Rham complex.
\end{lem}

\paragraph{Fat Forms: The Der-Complex}
In \cite{Rub01} the cohomology of an important algebraic object associated to a vector bundle was computed. We relate this result to the formalism of diolic calculus.

\begin{thm}
The degree one diolic de Rham complex coincides with the Der-complex for the $A$-module $P.$
\end{thm}
Let us demonstrate this result by constructing all necessary aspects systematically. 
\begin{lem}
The degree $1$ diolic $1$-forms
coincide with fat differential $1$-forms on $P.$
\end{lem}
By consider 
$\overline{\Lambda}^k(P):=\bigwedge^k\bar{\Lambda}^1(P),$ which is the module of fat $k$-forms given as
$\text{Hom}_A\big(\bigwedge^k\text{Der}(P),P\big),$ we see that in the graded setting of diolic algebras, they are precisely those alternating $A$-multilinear maps on $\text{Der}(P)$ of degree $1$. Denoting them accordingly by $\text{Alt}_A^k(\text{Der}(P);P),$ we find the  \emph{Der-sequence} of $P$ denoted by $\Lambda\mathfrak{D}\text{er}(P)^{\bullet}:$
\begin{equation}
\label{eqn:dercomplex}
0\rightarrow P\xrightarrow{\partial^0}\text{Hom}_A(\text{Der}(P),P)\xrightarrow{\partial^1}...\rightarrow\text{Alt}_A^i(\text{Der}(P),P)\xrightarrow{\partial^i}\text{Alt}_A^{i+1}(\text{Der}(P),P)\rightarrow...
\end{equation}
The differential is defined according to the rule,
\begin{equation*}
\partial^0:P\rightarrow \text{Hom}_A\big(\text{Der}(P),P)\big),
\hspace{1mm} p\mapsto \partial^0p
\end{equation*}
where $
\partial^0 p:\text{Der}(P)\rightarrow P,
\Delta\longmapsto \partial^0p(\Delta):=\Delta(p),$
and the higher order differentials are defined by extending linearly via the graded-Leibniz rule, yielding
\begin{eqnarray}
\partial^k\omega\big(\Delta_1,...,\Delta_{k+1}\big)&=&\sum_{i=1}^{k+1}(-1)^{i-1}\Delta_i\big(\omega(\Delta_1,..,\hat{\Delta}_i,..,\Delta_{k+1})\big) \nonumber
\\
&+&\sum_{i< j}(-1)^{i+j}\omega\big([\Delta_i,\Delta_j],\Delta_1,..,\hat{\Delta_i},..,\hat{\Delta_j},..,\Delta_{k+1}\big)
\end{eqnarray}
for $\omega\in \text{Alt}_A^k(\text{Der}(P),P),\Delta_1,..,\Delta_{k+1}\in \text{Der}(P).$

From an entirely similar computation to the case of the ordinary de Rham differential, one may verify that $\partial^2=0,$ so the Der-sequence for $P$ is in fact a complex, often called the Der-complex.
Denote the cohomology of (\ref{eqn:dercomplex}) by $H_{\text{Der}}(P).$
While providing a careful cohomological analysis of the natural complexes associated to diolic algebras is outside the scope of the present paper, we will describe some aspects of the cohomology theory associated to this degree $1$ diolic de Rham complex (\ref{eqn:dercomplex}).
\subsubsection{Der-cohomology}
The cohomology of \ref{eqn:dercomplex} may be computed by noting that in particular, $\text{Der}(P)$ is a Lie algebroid and the general theory of Lie algebroid cohomology is well-understood. Specifically, one may introduce a natural filtration and compute the cohomology by using a suitable spectral sequence.
To this end, we associate a natural filtration to the module $\Lambda\mathfrak{D}\text{er}^{\bullet}(P)$, and show the resulting spectral sequence converges to the above cohomology $H_{Der}(P).$
Let $\overline{X}_i\in\text{End}(P)\subset \text{Der}(P)$ and suppose that $\Omega\in\Lambda\mathfrak{D}\text{er}^{p+q}(P).$  We define the above mentioned filtration by declaring $\Omega$ to lie in the $p$'th piece, written $\Lambda_{(p)}\mathfrak{D}\text{er}^{p+q}(P)$ if $
    i_{\overline{X}_1}\circ..\circ i_{\overline{X}_{q+1}}\big(\Omega\big)=0,$
where $i_{\overline{X}}$ is the insertion operator. Clearly, one has a sequence
$..\Lambda_{(p+1)}\mathfrak{D}\text{er}^{p+q}(P)\subset \Lambda_{(p)}\mathfrak{D}\text{er}^{p+q}(P)\subset...$
The zero'th term of the corresponding spectral sequence is
\begin{equation}
\label{eqn:DerSpecSeq}
    E_0^{p,q}(P):=\frac{\Lambda_{(p)}^{p+q}\mathfrak{D}\text{er}^{p+q}(P)}{\Lambda_{(p+1)}^{p+q}\mathfrak{D}\text{er}^{p+q}(P)}.
\end{equation}
\begin{lem}
The associated spectral sequence to the filtration defined above is such that $E_0^{p,q}(P)\cong \Lambda^p(A)\otimes_A \Lambda^q(\mathrm{End}(P)),$ while $E_1^{p,q}(P)\cong \Lambda^p(A)\otimes_AH^q\big(\mathrm{End}(P)\big),$ and $E_2^{p,q}(P)\cong H_{dR}^p\big(A;H^q(\mathrm{End}(P))\big).$
\end{lem}

\begin{exm}
Let $\mathcal{P}$ be a principle $G$-bundle over $M$, with associated bundle $V,$ and let $\mathscr{A}_{\mathcal{P}}$ denote the artificial diole defined by this associated bundle. It follows from \cite{Rub01}, that the cohomology of the Der-complex $H_{Der}(P,A)$ is isomorphic to the algebra of de Rham cohomology $H(\mathcal{P})$ of this principle bundle.
\end{exm}

It is possible to generalize this situation slightly to allow coefficients in an arbitrary module $Q.$ To do this, we must make sense of the action of derivations of $P$ on $Q.$
This is done by introducing the notion of a \emph{derivation representation}, the details of which are available in \cite{Rub01}, for instance.
The following is a diolic rephrasing of this notion. 
Suppose that we have two diolic algebras, $\mathscr{A}_1:=A\oplus P_1$ and $\mathscr{A}_2:=A\oplus P_2,$ where $P_1,P_2\in\mathrm{Mod}(A)$.
\begin{definition}
A \emph{diolic derivation representation}\footnote{We may sometimes, in light definition \ref{diolicdiffrep} below, call these \textit{first order representations}} \emph{of} $D(\mathscr{A}_1)_{\mathcal{G}}$ in the diolic algebra $\mathscr{A}_2$ is a morphism of graded Lie algebras $f:D(\mathscr{A}_1)\rightarrow D(\mathscr{A}_2)$ of degree zero inducing a morphism of diolic Atiyah sequences.
\end{definition}
In particular, this definition implies that $f=f_0\oplus f_1$ with $f_0:D(\mathscr{A}_1)_0\rightarrow D(\mathscr{A}_2)_0$ a Lie algebra morphism such that the diagram of $A$-modules
\[
\begin{tikzcd}
0\arrow[r,]& \text{End}_A(P_1)\arrow[d, "f|_{\text{End}}"] \arrow[r,] & D(\mathscr{A}_1)_0 \arrow[d, "f"] \arrow[r,] & D(A)\arrow[d,"id"] \arrow[r,] & 0
\\
0\arrow[r,]& \text{End}_A(P_2)\arrow[r,] & D(\mathscr{A}_2)_0 \arrow[r,] & D(A) \arrow[r,] & 0
\end{tikzcd}
\]
commutes.
Fixing splittings of the Atiyah sequences identifies the middle terms with the modules of Der-operators in $P_1,$ and $P_2$ respectively, so that the above definition coincides with the usual one as in \cite{Rub01}.
In this case, the construction of the above spectral sequence holds, and coincides with Lie-algebroid cohomology with values in a representation, $Q.$ 
\subsection{Diolic bi-derivations and canonical structures}
\label{ssec:DioleCanStrctr}
We will now compute the multi-derivations for the diolic algebra $\mathscr{A}.$ Specifically, we study the bi-derivations, using the functorial definition (\ref{eqn:BiDer}).

\begin{rmk}
We fix our notation for a graded bi-derivation of grading $i$, as $\Pi_i$ which acts on homogeneous components $j,k$ with component $\Pi_{i}^{j,k}.$ We will envoke proposition \ref{biderdual} and always denote a bi-derivation as $\tilde{\Pi}$ and the corresponding to it derivation as $\Pi$ under the rule
$\tilde{\Pi}(a,b)=\big(\Pi(a)\big)(b).$ 
\end{rmk}
By definition of the diolic algebra, there will be diolic biderivations only in degrees $-2,-1,0,1.$ This can easily be justified by the functorial definition of $D_{2}^{\mathcal{G}}.$

\begin{lem}
The graded biderivations of degree $1$ of $\mathscr{A}$ are given by $P$-valued bi-derivations of $A$.
\end{lem}

Turning now to compute those graded-bi-derivations of degree $-2$, we see by definition of such operators, we must be in the situation when $P$ is a module of rank $1.$ In this case, we have the following un-graded description of degree $-2$ Poisson structures.
\begin{lem}
A Poisson structure of degree $-2$ on the algebra of dioles is nothing but an $A$-bilinear, $A$-valued skew-symmetric form on $P.$
If $P$ is not of rank $1$, there are no such graded derivations.
\end{lem}
Consequently, we see in the special instance of treating line bundles, the formalism for diolic calculus permits graded biderivations of degree $-2.$ 

Passing to the next grading, we see that diolic bi-derivations of degree $-1$ deserve special attention.

\begin{lem}
A graded biderivation of degree $-1$ on the algebra of dioles is a $\mathrm{Der}(P)$-valued diolic derivation of degree $-1$. That is $\Pi_{-1}\in \mathrm{Hom}_A\big(P,\mathrm{Der}(P)\big).$
\end{lem}
This result admits an interesting interpretation as follows.
Recall that a Poisson bracket on a graded algebra $\mathscr{A}$ is a Lie bracket which is additionally a bi-derivation. Let $g$ denote the degree of the bracket so that its homogeneous action is given as $\big\{.,.\big\}_g:\mathscr{A}_i\times \mathscr{A}_j\rightarrow \mathscr{A}_{i+j+g}.$
With this in mind, we see the following.
\begin{lem}
\label{DioleLieAlgbd}
Let $\mathscr{A}$ be a diolic algebra. Let $\{.,.\}_{-1}$ be a degree $-1$ Poisson bracket on $\mathscr{A}.$ Let $\alpha:P\rightarrow D(A),$ be the $A$-homomorphism defined by $p\longmapsto \alpha(p):=\{p,.\}_{-1}\big|_A.$
Then the triple 
$\big(\mathscr{A},\big\{.,.\big\}_{-1},\alpha\big),$
is a Lie algebroid.
\end{lem}
\begin{proof}
We need to show that $\alpha(bp)=b\alpha(p)$ for all $b\in A$ so that $\alpha\in\text{Hom}_A\big(P,D(A)\big)$ is really an $A$-homomorphism, that we have the Leibniz rule
$\{p,bq\}_{-1}=b\{p,q\}_{-1}+\{p,b\}_{-1},$
 with $\alpha(p)=\{p,\}_{-1}|_A,$ and that $\alpha$ is a homomorphism of Lie algebras. The latter condition obviously means
$\alpha\big(\{p_1,p_2\}_{-1}\big)(b)=\big[\alpha(p_1),\alpha(p_2)\big](b),$
for all $p_1,p_2\in P,b\in A.$ 
We get these by defining 
$\alpha(p):=\big\{p,.\big\}_{-1}\big|_A:A\rightarrow A,$
and using the bi-derivation property of the Poisson bracket to show this is a derivation. Similarly the Jacobi identity for the Poisson bracket shows that this is in fact a Lie algebra homomorphism. Moreover by setting $\{.,.\}_{-1}\big|_P:=[.,.]_P,$ we recover the formula for the Leibniz rule for Lie algebroids with anchor $\alpha.$
\end{proof}
This result tells us that Lie algebroids are nothing but Poisson
structures on algebras of dioles of degree $-1$, so we have that the theory of algebroids finds a natural home in the
context of differential calculus over diole algebras.

We now compute the final component of $D_{2}(\mathscr{A})_{\mathcal{G}}.$

\begin{lem}
Degree zero graded bi-derivations of $\mathscr{A}$ are given by pairs  $\Pi_0=(\Pi_{0}^{AA},\Pi_{0}^{AP}),$
where $\Pi_{0}^{AA}\in D_2(A)$ is an ordinary bi-vector, and $\Pi_{0}^{AP}:A\times P\rightarrow P$ with the two operators related by the following Der-Leibniz like rule,
$$\Pi_{0}^{AP}(a,bp)=\Pi_{0}^{AA}(a,b)p+b\Pi_{0}^{AP}(a,p),$$
where $a,b\in A,p\in P.$
\end{lem}
Let us give an interesting interpretation of the
above degree $0$ diolic biderivations. 
Write $\tilde{\Pi}_0^{A-}:=\{a,-\}_0,$ in terms of a Poisson bracket. Then, $\tilde{\Pi}_0^{A,-}(a,-)\in \text{Der}(P),$ which is to say, a Der-operator in the second entry. The corresponding Der-Leibniz rule, yields 
$$\tilde{\Pi}_0^{AP}(a,bp)=\{a,b\}p+b\tilde{\Pi}_0^{AP}(a,p),$$
where we wrote $\tilde{\Pi}_0^{AA}(a,b)=\Pi_0^A(a)(b)=X_a(b),$ for $X_a:=\{a,-\}_0$ a Hamiltonian vector field.
Since Der-operators are defined in terms of the splitting of the Atiyah sequence via a flat connection, these bi-derivations are thus determined by a Poisson structure $\{-.-\}_0$ on $A$ and a \textit{flat Hamiltonian connections}; those Der-operators with Hamiltonian symbols.

\subsubsection{Multi-derivations}
We quickly present the result for arbitrary order derivations. That is, we describe what objects of 
 $D_{k+1}(-):=D\big(D_k(-)\subset \mathfrak{D}^>_k(-)\big),$ look like for the diolic algebra $\mathscr{A},$ in the geometric setting where $A=C^{\infty}(M)$ for a smooth $n$-dimensional manifold $M$, and where $P$ is the module of sections of vector bundle $E\rightarrow M.$ 
\begin{lem}
An arbitrary graded $k$-multiderivation of the diolic algebra $\mathscr{A}$ has components, 
$\Delta_{-k-i+1}^{\mathsmaller{A,..,A,P,..,P}}\in D_i\big(\mathrm{Der}_{k-i}(P)\big)\cong \mathrm{Sym}^i(T_M)\otimes_{C^{\infty}(M)}D_{k-i}(E)$ where $A$ appears $k$ times $P$ appears $(i-1)$-times, and 
$\Delta_{-k-i+1}^{\mathsmaller{A,..,A,P,.,P}}\in D_i\big(A,\mathrm{Alt}_{k-i}(P,A)\big)\cong \mathrm{Sym}^i(T_M)\otimes_{C^{\infty}(M)}\mathrm{Sym}^k(E^{\vee})$ where $A$ appears $i$-times and $P$ appears $k-1$-times.
\end{lem}

\begin{exm}
A $3$-derivation $\Delta\in D_3(\mathscr{A})$ is given by a decomposition
$\Delta=\Delta_{-3}+\Delta_{-2}+ \Delta_{-1}+ \Delta_0+ \Delta_1,$
with homogeneous components
$$
 \scalemath{0.83}{
\begin{cases}
\Delta_{-3}^{PPP}\in\text{Alt}_3(P,A)
\\
\Delta_{-2}^{PPP}\in\text{Der}_3(P),
\\
\Delta_{-2}^{APP}\in D\big(A,\text{Alt}_2(P,A)\big)
\\
\Delta_{-1}^{APP}\in D_2(A,P^{\vee}),
\\
\Delta_{0}^{AAA}\in D_3(A,A)
\\
\Delta_0^{AAP}\in D_2\big(A,\text{Der}(P)\big) 
\\
\Delta_1^{AAA}\in D_3(A,P),
\end{cases}}.
$$
\end{exm}
\subsection{Diolic canonical structures}
We now turn our attention to study those bi-derivations of the diolic algebra with the property that $[\![\Pi,\Pi]\!]=0.$
The latter is an element of $D_3(\mathscr{A})$, and by the description of such graded $3$-derivations, we see that this relation is tautologically satisfied for degree reasons when $\Pi$ is of degree $-2.$ By similar degree reasons, the non-trivial grading which we must focus on is for those operators in degree $0.$

\subsubsection{Diolic Poisson structures of degree $0$}
Applying the inductive definition of the graded Schouten bracket, we can deduce those diolic Poisson structures of interest.
\begin{thm}
\label{diolePDE}
Let $\Pi_0=(\Pi_0^{AA},\Pi_0^{AP})\in\text{D}_{2}(\mathscr{A})_0,$ be a degree $0$ diolic-bi-derivation. Then it is a diolic Poisson structure, ie. $[\![\Pi_0,\Pi_0]\!]=0,$ if and only if $[\![\Pi_0^{AA},\Pi_0^{AA}]\!]=0,$ and
$
 \Pi_0^{AP}\big(a_0,\Pi_0^{AP}(a_1,p_0)\big)-\Pi_0^{AP}\big(a_1,\Pi_0^{AP}(a_0,p_0)\big)-\Pi_0^{AP}\big(\Pi_0^{AA}(a_0,a_1),p_0\big)=0,$
for all $a,b\in A,$ and $p\in P.$
\end{thm}
This result tells us that a degree zero diolic-biderivation is a Poisson structure if and only if its components satisfy a pair of non-linear PDEs.
\subsubsection{Coordinates}
Consider for simplicity the case of a trivial rank $m$-bundle on an $n$-dimensional manifold $M.$ Let $x^i$ be coordinates for $A$ and $u_{\alpha}$ be those for $P$. Then the module $\text{Der}(P)$ is generated by $\big<\nabla_i, u_{\alpha}^{\beta}\big>.$ That is, for $p=p^{\alpha}e_{\alpha}\in P$, we have basis elements in $\text{Der}(E)$ defined by $\nabla_i(p)=\partial_i p^{\alpha}e_{\alpha}$ and similarly, $u_{\beta}^{\alpha}:=p^{\alpha}e_{\beta}$, all of course locally near a point $x\in M.$ We find a diolic bi-derivation of degree zero locally is given by
$\Pi_0^A=\big(\Pi^{i,j}\nabla_j+\Pi_{\beta}^{i\alpha}u_{\alpha}^{\beta}\big)\otimes \frac{\partial}{\partial x_i}.$

Computing the quantity $[\![\Pi_0,\Pi_0]\!](a,b,c)=0$ on triples $a,b,c\in A,$ one finds for
$\overline{\Pi_0}:=\Pi_0^{AA}$ the zero component of $D_{2}(\mathscr{A})_0,$ which is to say $D_2(A),$ the `usual' Poisson PDE
\begin{equation}
    \sum_{\ell=1}^m\bigg(\overline{\Pi}^{i\ell}\partial_{\ell}\overline{\Pi}^{jk}+\overline{\Pi}^{j\ell}\partial_{\ell}\overline{\Pi}^{ki}+\overline{\Pi}^{k\ell}\partial_{\ell}\overline{\Pi}^{ij}\bigg)=0,
    \end{equation}
    where $\partial_{\ell}$ is of course shorthand for $\frac{\partial}{\partial x_{\ell}}.$
    
    Computing in a similar way the quantity $[\![\Pi_0,\Pi_0]\!](a,b,p)=0,$ which reads 
  $\Pi_0^A(a)\big(\Pi_0^A(b)(p)\big)-\Pi_0^A(b)\big(\Pi_0^A(a)(p)\big)=\Pi_0^A\big(\overline{\Pi_0}(a,b)\big(p),$
for isntance, in the case when $a=x^kb=x^{\ell}p=u_{\gamma},$ yields the following PDE 
\begin{equation}
\label{eqn:DiolePoissPDE}
\sum_{i=1}^n\bigg(\Pi_{\beta}^{i\alpha}\frac{\partial}{\partial x_i}\Pi^{jk}+\Pi^{ij}\frac{\partial}{\partial x_i}\Pi_{\beta}^{k\alpha}-\Pi^{ik}\frac{\partial}{\partial x_i}\Pi_{\beta}^{j\alpha}\bigg)=0.
\end{equation}
We call (\ref{eqn:DiolePoissPDE}) the \emph{diolic Poisson equation}, and any degree zero diolic bi-derivation whose components satisfy the usual Poisson-PDE and the diolic Poisson equation simultaneously, is a diolic Poisson structure.

\begin{rmk}
Note that any (graded) canonical structure of Poisson type defines a differential $\partial_{\Pi}:=[\![\Pi,-]\!]$ which naturally determines a complex $\big(D_{*}(\mathscr{A}),\partial_{\Pi}\big)$, referred to as the \emph{Poisson complex} of $\mathscr{A}$. One may then continue the analysis of diolic Poisson structures by studying the associated cohomology theory and computing the so-called \emph{diolic Poisson cohomologies}. 
\end{rmk}

\subsection{Diolic differential operators}
We now compute the graded modules of diolic differential operators.
\begin{thm}
\label{diolediffops}
Let $\mathscr{A}$ be a diolic algebra. A degree $0$ and order $k$, graded differential operator from $\mathscr{A}$ to $\mathscr{A}$ is a pair $\Box_0= (\Box_0^A,\Box_0^P)$ of operators with $\Box_0^A\in \mathrm{Diff}_k(A,A),\Box_0^P\in\mathrm{Diff}_k(P,P)$ satisfying 
$\delta_a^k(\Box_0^A)=\delta_a^k(\Box_0^P)$
for all $a\in A$ (ie. they have the same scalar-type symbol).
Those graded differential operators of degree $1$ and order $k$ are precisely $P$-valued differential operators of order $k$ on $A$. That is, we have an isomorphism of $A$-modules,
  $\mathrm{Diff}_{k}(\mathscr{A})_1\cong \mathrm{Diff}_k(A,P).$
  Moreover, when $P$ is a projective module of rank $1$ there exists degree $-1$ diolic differential operators which are elements of $\mathrm{Diff}_k(P,A).$
\end{thm}
\begin{proof}
The proof is by induction on the order $k.$ 
Consider the degree zero case for order $1.$ Such an operator is a degree zero $\mathbb{K}$-linear morphism $\Box_0:\mathscr{A}\rightarrow \mathscr{A}$ determined by its restriction to homogeneous components of $\mathscr{A},$ such that it satisfies the definition of being a first order graded differential operators, which means
$$\delta_{a_0,a_1}(\Box_0)=0,$$
for homogeneous $a_0,a_1\in \mathscr{A}.$ 
The components, explicitly are written as
  $\Box_{0}^A:A\rightarrow A$ and $\Box_{0}^P:P\rightarrow P$ and we compute the admissable cases for $\delta_{a_0,a_1}.$
  For $a,b\in A$ the expression $\delta_{a,b}(\Box_0)$ acting on $1_A\in A$ or $p\in P$, yields $\Box_0^A,\Box_0^P$ are ordinary differential operators $A\rightarrow A$ and $P\rightarrow P,$ respectively.
  The condition we see above comes from computing 
  $\delta_{a,p}(\Box_0)(b)=0,$
  which we can, in particular, computed for $b=1_A,$ which yields exactly that $p\delta_a(\Box_0^A)=\delta_a(\Box_0^P)(p),$
  where on the left hand side of this expression, we interpret $p:A\rightarrow P$ as the degree $1$ order zero differential operator of `multiplication by $p\in P.$
  Proceeding inductively one finds our result. It is based on the fact that $\delta_a^{k+1}=\delta_{a}\circ \delta^k$ and for degree purposes we can have at most one degree $1$ homogeneous element $p\in P$ in the tuple of elements $\delta_{a_0,a_1,..,a_{k-1},p}$. In particular, we assume the case $a_0=a_1=..=a\in A.$ 
The above description of degree one differential operators is easily established. Simply note that for degree reasons such an operator is determined by its components
$\Box_{1}^A:A\rightarrow P$ and $\Box_{1}^P:P\rightarrow 0$
 with $\Box_1^P\equiv 0$ identically zero. The rest is immediate.
\end{proof}

In the same way a degree zero derivation of a diole algebra coincided geometrically with a Der-operator, suitably interpreted as a covariant derivative associated to a connection, the description of degree zero differential operators in theorem \ref{diolediffops} suggests that a degree zero diolic differential operator is really some kind of operator which behaves like a higher order covariant differential. To better formalize this, we turn our attention to finding an analogue of the Atiyah sequence for differential operators of order $k>1$.
Indeed, for a degree zero differential operator $\Box_0$, we have, in analogy with the symbol map $\text{Der}(P)\xrightarrow{\sigma}D(A),$ a projection 
$$\varsigma:\text{Diff}_{k}(\mathscr{A})_0\rightarrow \text{Diff}_k(A),$$
which sends $\Box_0=(\Box_{0}^A,\Box_{0}^P)$ to its component $\Box_{0}^A.$

\begin{lem}
There is an isomorphism of $A$-modules
$\ker(\varsigma)\cong \mathrm{Diff}_{k-1}(P,P).$
\end{lem}

We thus find the analogous sequence for higher order degree zero graded diolic operators, which in the case when $P$ is a projective module, is short exact:
\begin{equation}
\label{eqn:diAtk}
0\rightarrow \text{Diff}_{k-1}(P,P)\hookrightarrow \text{Diff}_{k}(\mathscr{A},\mathscr{A})_0\xrightarrow{\varsigma}\text{Diff}_k(A,A)\rightarrow 0.
\end{equation}
We call these the \emph{diolic Atiyah sequences of order} $k$, and denote them by $\text{At}_k(\mathscr{A}).$
Note that for each $\ell\leq k$, we have an embedding $\text{At}_{\ell}(\mathscr{A})\hookrightarrow \text{At}_k(\mathscr{A}).$
The \emph{infinite order diolic Atiyah sequence} $\text{At}(\mathscr{A})$
is the direct limit of natural embeddings $..\subset \text{At}_k(\mathscr{A})\subset \text{At}_{k+1}(\mathscr{A})\subset..$

The utility of such sequences for computation purposes is apparent. Just as we used a splitting of $\text{At}_1(\mathscr{A})$ to find a coordinate description of Der-operators in $P$, we may look for similar splittings of $\text{At}_k(\mathscr{A})$ for each $k$.

\subsubsection{Higher order connections in an algebra of dioles}
In the spirit of the Atiyah sequence (for order $k=1$), where the space of splittings is the space of linear connections in the module $P$, we can analyze the space of splittings for higher order Atiyah sequences (\ref{eqn:diAtk}). Consider $\mathscr{Q}$ a graded $\mathscr{A}$-module. In the left $\mathscr{A}$-module $\text{Diff}_k^<(\mathscr{A},\mathscr{Q})_{\mathcal{G}}$ there is a graded-submodule with respect to this structure denoted 
$$\text{Dif}_k^<(\mathscr{A},\mathscr{Q})_{\mathcal{G}}:=\big\{\Box\in \text{Diff}_k(\mathscr{A},\mathscr{Q})_{\mathcal{G}}|\Box(1_\mathscr{A})=0\big\}.$$
Notice that in the $k=1$ case, we recover precisely graded derivations.
\begin{rmk}
These functors play a role in studying higher order analogues of graded differential forms, and provide a conceptual environment for studying the so-called \textit{higher order de Rham complexes}, as studied in the ungraded case in \cite{VezVin}. Roughly speaking, using such functors leads us to a formalism for calculus in which the de Rham differential need not be a first order differential operator.
\end{rmk}

If we 'restrict' the above diolic Atiyah sequence of order $k$ to the submodules of the form $\mathrm{Dif}_k^<(\mathscr{A})_0\subset \mathrm{Diff}_k(\mathscr{A})_0$ we recover exactly the appropriate notion of linear connections for the $k=1$ case.

This is an obvious consequence of our formalism for differential calculus over diolic algebras. 
It is immediate from our characterization of order one diolic differential operators of degree $0$ and that our diolic Atiyah sequence of order $1$ coincides with the Atiyah sequence of the vector bundle $P$, whose space of splittings is precisely the space of linear connections in $P.$ 

\begin{rmk}
A more satisfactory categorical way to view this sub-object is given by viewing it as an end-functor $\text{Dif}_k^<(-)_{\mathcal{G}}\in \text{Fun}\big(\text{Mod}^<(\mathscr{A})_{\mathcal{G}}, \text{Mod}^<(\mathscr{A})_{\mathcal{G}}\big)$, together with a short exact sequence in this category, as
$0\rightarrow \text{Dif}_k(-)_{\mathcal{G}}\xrightarrow{i_k}\text{Diff}_k^<(-)_{\mathcal{G}}\xrightarrow{\mathcal{D}_k}\mathbb{I}_{\text{Mod}(\mathscr{A})}\rightarrow 0,$
where the natural transformation $\mathcal{D}_k$ is, given on a fixed $\mathscr{A}$-module $\mathscr{P}$ as the morphism $\mathcal{D}_k(\mathscr{P})\colon \Delta\in \text{Diff}_k(\mathscr{A},\mathscr{P})\rightarrow \Delta(1)\in \mathscr{P}.$
\end{rmk}
Togther with the functors of differential calculus, the following is a natural definition in the diolic formalism.

\begin{definition}
A \emph{diolic connection of order} $k$, or simply a \emph{diolic} $k$\emph{-connection}, is an $\mathscr{A}_0$-linear splitting of the $k$'th order diolic Atiyah sequence $(\ref{eqn:diAtk}),$ which factors through the sub-object $\mathrm{Dif}_k.$
\end{definition}
Therefore, the definition of a diolic $k$-connection amounts to an $A$-linear homomorphism,
$\nabla^{(k)}:\text{Diff}_k(A,A)\rightarrow \text{Diff}_k(\mathscr{A},\mathscr{A})_0,$
such that $\nabla^{(k)}(f^<\Box+g^<\Box')=f^<\nabla_{\Box}^{(k)}+g^<\nabla_{\Box'}^{(k)}, f,g\in A, \Box,\Box'\in \text{Diff}_k(A),$ and this map is a splitting, ie. $\varsigma\circ \nabla^{(k)}=id_{\text{Diff}_k(A,A)}.$
Moreover, the property of factoring through the sub-object amounts to requiring that we have $\nabla^{(k)}(\Box)(1)\equiv 0.$
The condition of being a splitting means that for all $\Box\in \text{Diff}_k(A,A)$ we have $\varsigma\big(\nabla_{\Box}^{(k)}\big)=\Box,$ and so we call the operator $\nabla_{\Box}^{(k)}$ the \textit{diolic covariant differential of order} $k$ along $\Box$ with respect to the diolic $k$-connection.
If we work with $\text{Dif}$ rather than $\text{Diff}$ we may omit the last condition. In the case of implementing a suitable restriction, it is rather clear that every linear connection in a vector bundle $\pi:E\rightarrow M$ is a diolic connection.

\begin{rmk}
Obviously a naive definition of a diolic $k$-connection which is \emph{flat} does not make sense as there is no Lie algebra structure on $\text{Diff}_k;$ however, in the direct limit $\text{Diff}$, one may try to formalize this notion using the graded Lie algebra structure given by the graded commutator.
\end{rmk}
The existence of higher order diolic Atiyah sequences suggests also a natural extension of the definition of a derivation representation. That is, with such a sequence at our disposal it is natural to make the following definition within the formalism of diolic differential calculus for two such diolic algebras $\mathscr{A}_1,\mathscr{A}_2$.

\begin{definition}
\label{diolicdiffrep}
A \emph{differential representation} of $\mathrm{Diff}(\mathscr{A}_1,\mathscr{A}_1)_0$ in the diolic algebra $\mathscr{A}_2$ is a commutative diagram
\[
\begin{tikzcd}
0\arrow[r,]& \mathrm{Diff}(P_1,P_1)\arrow[d, "f"] \arrow[r,] & \mathrm{Diff}(\mathscr{A}_1)_0 \arrow[d, "f"] \arrow[r,] & \mathrm{Diff}(A)\arrow[d,"id"] \arrow[r,] & 0
\\
0\arrow[r,]& \mathrm{Diff}(P_2,P_2)\arrow[r,] & \mathrm{Diff}(\mathscr{A}_2)_0 \arrow[r,] & \mathrm{Diff}(A) \arrow[r,] & 0
\end{tikzcd}
\]
such that $f$ is a Lie-algebra morphism as well as a morphism of $A$-modules.
\end{definition}
In the diagram appearing in the above definition, the left most vertical arrow is again an appropriate restriction map.  

\subsubsection{Module structures}
In this section we will investigate the module structures which naturally exist on $\mathrm{Diff}(\mathscr{A})$, the collection of all diolic differential operators. 

For computation purposes it is useful to view in components the module structure as well as how composition of diolic operators acts.
Clearly as each homogeneous component of degree $0,1$ given as $\text{Diff}(\mathscr{A},\mathscr{A})_0\cong \text{Diff}(P,P)\oplus \text{Diff}(A,A),$ we have a well-defined $A$-module structure, since each component is an $A$-module.  Similarly for $\text{Diff}(\mathscr{A},\mathscr{A})_1=\text{Diff}(A,P).$ 
We turn our attention to the $\mathscr{A}_1=P$-module structure.
Let $\Box_0=(\Box_0^A,\Box_0^P)$ be a degree zero diolic operator.

We want to understand the quantity $p^<\cdot \Box.$ Clearly, this will be a diolic operator of degree $1.$ We see that such an operator acts trivially on the component $P$, so we find the required left $P$-module structure
$$(p^<\cdot \Box_0)(a):=p\cdot \Box_{0}^A(a).$$
In this expression, we interpret $p$ as the degree $1$ graded differential operator of order $0.$ 
The right module-structure is similarly defined
$(p^>\cdot \Box)(a):=(-1)^{p\cdot \Box}(\Box_0^P\circ p)(a)=\Box_0^P(pa).$
We note that the differential operator $\Delta_0^P\circ p$ is again an operator of order $k.$
\begin{lem}
The module of diolic differential operators is a diolic module. 
\end{lem}

Let us provide an explicit description of the diolic structure morphism. It is the $A$-module homomorphism 
$$\beta_{\text{Diff}_k}:P\otimes_A\text{Diff}_k(\mathscr{A},\mathscr{A})_0\rightarrow \text{Diff}_k(A,P),\hspace{2mm} \beta_{\text{Diff}}(p\otimes \Delta_0):=p\circ \Delta_0.$$ For such a degree zero operator $\Delta_0=(\Delta_0^A,\Delta_0^P)$ this composition is defined to be $p\circ \Delta_0^A.$ The latter is clearly a $P$-valued differential operator on $A$ of order $k.$ 

\subsubsection{Operators with values in Diolic modules}
We compute precisely in the same manner as above, the more general case of $\mathscr{R}$-valued diolic differential operators.
They are described as above but with an additional twisting of the symbol maps by the diolic structure morhism.
\begin{thm}
A $k$th order diolic differential operator of order zero with values in the truncated diolic module $(\mathscr{R},\beta)$ is a pair of operators $\Box_0=(\Box_0^A,\Box_0^P),$ with $\Box_0\in \mathrm{Diff}_k(A,R_0),$ while $\Box_0^P\in\mathrm{Diff}_{k}(P,R_1),$ such that the following relation holds for all $p\in P,$
$\beta\big(p\otimes \delta_a^k(\Box_0^A(1_A)\big)=\delta_a^k(\Box_0^P)(p).$
Degree $1$ operators of order $k$ are those $\Delta_1^A\in \mathrm{Diff}_k(A,R_1).$
\end{thm}

We obtain the twisted symbol map $\sigma_k^{\beta}:\text{Diff}_k(\mathscr{A},\mathscr{R})_0\rightarrow \text{Diff}_k(A,R_1),$ which sends $\Box_0$ to $\beta_p^{\sharp}\circ \Box_0^A.$ This holds for all $p$ and consequently we work with the projection map $\sigma_k:\mathrm{Diff}_k(\mathscr{A},\mathscr{R})_0\rightarrow \mathrm{Diff}_k(A,R_0).$ 
\begin{lem}
Suppose $\mathscr{R}$ is such that the structure morphism $\beta$ is non-degenerate. Then
$\ker(\sigma_k)\cong \mathrm{Diff}_{k-1}(P,R_1)$ and we have a sequence
$0\rightarrow \mathrm{ker}(\sigma_k^{\beta})\cong \mathrm{Diff}_{k-1}(P,R_1)\rightarrow \mathrm{Diff}_k(\mathscr{A},\mathscr{R})_0\rightarrow\mathrm{Diff}_k(A,R_0),$
for all $\Box_0^A$ operator such that $\delta_a^k(\Box_0^A(1_A))$ is not contained in $\mathrm{ker}(\beta^{\sharp}).$
\end{lem}

Moving forward, we will need also the most general situation. Consider $(\mathscr{R},\beta)$ and $(\mathscr{S},\gamma)\in \tau\text{Mod}(\mathscr{A}).$ For simplicity assume these are non-degenerate truncated diolic modules. Then we have the following description of $\text{Diff}\big(\mathscr{R},\mathscr{S})_{\mathcal{G}}.$

\begin{lem}
A $k$'th order diolic differential operator $\Delta_0:(\mathscr{R},\beta)\rightarrow (\mathscr{S},\gamma),$ between two truncated diolic modules is a pair of operators $\Delta_0:=\big(\Delta_0^{R_0},\Delta_0^{R_1}\big)$ where $\Delta_0^{R_0}\in \mathrm{Diff}_k(R_0,S_0),$ is an ordinary differential operator between the $A$-modules $R_0,S_0,$ and where $\Delta_0^{R_1}\in \mathrm{Diff}_k(R_1,S_1)$ which satisfies the relation
$\gamma_p^{\sharp}\circ \delta_a^{k}\Delta_0^{R_0}=\delta_a^{k}\Delta_0^{R_1}\circ \beta_p^{\sharp}.$
\end{lem}

\subsection{Degree $0$ Diolic Jacobi structures}
We now perform a similar analysis of sub-section \ref{ssec:DioleCanStrctr}, but for the functors bi-differential operators (\ref{eqn:BiDiffops}).
\vspace{1mm}
\begin{definition}
A \emph{diolic Jacobi structure} is a first order bi-differential operator  $\Delta_g$ such that $[\![\Delta_g,\Delta_g]\!]^{SJ}=0.$
\end{definition}
Let $\Box_0\in\text{Diff}_1^{(2)}(\mathscr{A},\mathscr{A})_0.$ This means that 
$\Box_0\in\text{Diff}_1\big(\mathscr{A},\text{Diff}_1(\mathscr{A},\mathscr{A})\big)_0,$ which amounts to a description via homogeneous components as
$$\scalemath{0.85}{\Box_0:=\begin{cases}
\Box_0^A:A\rightarrow \text{Diff}_1(\mathscr{A},\mathscr{A})_0\in \text{Diff}_1\big(A,\text{Diff}_1(\mathscr{A})_0\big),
\\
\Box_0^P:P\rightarrow \text{Diff}_1(A,P)\in \text{Diff}_1\big(P,\text{Diff}_1(A,P)\big),
\end{cases}}
$$
which are related by 
$\beta_{diff}\big(p\otimes \delta_a(\Box_0^A(1_A))\big)=\delta_a(\Box_0^P)(p),$ for $a\in A,p\in P.$
Putting this relation aside for now, we see the above operators are equivalent to 
$\tilde{\Box}_0^{AA}:A\rightarrow \text{Diff}_1(A,A),$ and 
$\tilde{\Box}_0^{AP}:A\rightarrow \text{Diff}_1(P,P),$
and $\tilde{\Box}_0^{PA}:P\times A\rightarrow P.$
Equivalently, these are operators
$\tilde{\Box}_0^{AA}:A\times A\rightarrow A,$ and $\tilde{\Box}_0^{AP}:A\times P\rightarrow P,$ with the property that they have the same scalar-type symbol in the first entry, and by skew-symmetricity, in the second also.
By computing $[\![\Box_0,\Box_0]\!]^{SJ}(a_1,a_2,a_3)=2\text{Jac}_{\Box_0}^{3}(a_1,a_2,a_3)=0,$ for $a_1,a_2,a_3\in A,$ we find the Jacobi identity for $\tilde{\Box}_0^{A,A}$.

This means that $\tilde{\Box}_0^{AA}$ is an ordinary Jacobi structure on the manifold $M=\text{Spec}(A).$ In other words, it determines a Jacobi-structure on the trivial rank $1$ line bundle $\mathbb{R}_M$.
For $\tilde{\Box}_0^{AP}$ we have something a little different. Computing the quantity, $\text{Jac}_{\Box_0}^3(a,b,p)\equiv 0$, one finds
 \begin{equation}
 \label{DiolicJacobiPDE}
 \tilde{\Box}_0^{AP}\big(\tilde{\Box}_0^{AA}(a,b),p\big)-\tilde{\Box}_0^{AP}\big(a,\tilde{\Box}_0^{AP}(b,p)\big)+\tilde{\Box}_0^{AP}\big(b,\tilde{\Box}_0^{AP}(a,p)\big)=0.
 \end{equation}
In conclusion we have the following result.
\begin{lem}
A degree zero first order-bi-differential operator on $\mathscr{A},\Box_0$ is a  diolic Jacobi structure if and only if it satisfies the system of PDE's arising from the pair of Jacobiators
$\text{Jac}_{\Box_0}^3(a,b,c)=0,
\text{Jac}_{\Box_0}^3(a,b,p)=0,$
for all $a,b\in A,p\in P.$
\end{lem}
This tells us that $\Box_0$ is a Jacobi structure in a diolic algebra if and only if the component $\Box_0^{AA}$ determines a Lie algebra structure in $A$, interpreted as a Jacobi bracket on the trivial line bundle $\mathbb{R}_M$, and if the components $\Box_0^{AA}$ and $\Box_0^{AP}$ satisfy the compatibility equation  (\ref{DiolicJacobiPDE}), which we call the \textit{Diolic-Jacobi PDE}.

\subsubsection{Degree $-1$ Diolic bi-differential operators and Jacobi geometry}
Suppose we are in the geometric situation when $P$ is given by the module of sections of a line bundle $L\rightarrow M.$ Then by lemma \ref{deg-1derivations}, our diolic calculus admits degree $-1$ differential operators. They admit an un-graded description as those operators $P\rightarrow A,$ which are additionally $A$-linear. 
\begin{lem}
A diolic bi-differential operator of degree $-1$ is an operator
$\tilde{\Box}_{-1}^{PP}:=\{-,-\}_{-1}:P\times P\rightarrow P,$ that is a skew-symmetric and bi-differential of first order in each entry with the property that $\{p,-\}_{-1}$ is a first order differential operator on $P$ (so a derivation) with scalar-type symbol.
\end{lem}

Let us provide an alternative description of this result. Recall that a representation of a Lie algebroid $(A,[-,-]_A,\rho_A)$ in a  vector bundle $E$ is an $\mathbb{R}$-linear map $\Gamma(A)\rightarrow \text{Diff}_1(E,E)$ which satisfies $\nabla_{f\alpha}(e)=f\nabla_{\alpha}(e)$ and $\nabla_{\alpha}(fe)=\rho(\alpha)(f)e+f\nabla_{\alpha}(e),$ and is moreover, a Lie algebra morphism. 
A \emph{Jacobi algebroid} $(A,L)$ is the data of a Lie algebroid $A=\big([-,-],\rho\big)$ and a line bundle $L$ with a representation $\nabla$ of $A$, often referred to by the triple $\big([-,-],\rho,\nabla\big)$.

We then have the following interpretation of diolic Jacobi structure of degree $-1.$
\begin{lem}
Let $\{-,-\}_{-1}$ be a degree $-1$ Lie bracket on $\mathscr{A}:=C^{\infty}(M)\oplus \Gamma(L)$, with $L$ a line bundle. Suppose further that $\{-,-\}_{-1}$ is additionally a first order bi-differential operator (ie. a Jacobi bracket). Then this datum equivalently defines a unique Jacobi algebroid structure on the pair $(J^1L,L),$ which satisfies $\big[j_1(p_1),j_1(p_2)\big]=j_1\big(\nabla_{j_1p_1}p_2\big).$
\end{lem}

\subsubsection{Coordinates}
\label{sssec:Coords}
Let us assume we are in the smooth geometric situation with $\pi:E\rightarrow M$ a rank $m$ vector bundle over an $n$-dimensional manifold $M.$ If $(U,x_1,..,x_n)$ is a local coordinate system on $M$ with $\pi$ locally trivial with typical fiber $V$ with corresponding basis $e_{\alpha}$ where $\alpha=1,..,m$ we see that $P=\Gamma(\pi)$ is necessarily projective over $C^{\infty}(M)=A$ and so our sequence above is exact and the local splitting implies that for $\Box_0\in\text{Diff}_{k}(\mathscr{A})_0,$ we have
\begin{equation}
    \Box_{0}=\pmb{\Box}_{0}^A+||\Box_{i,j}||
    \end{equation}
where $\pmb{\Box}_{0}^A=\text{diag}(\Box_{0}^A,..,\Box_{0}^A)$ is a diagonal matrix of scalar operators $\Box_{0}^A=:\Box,$ and in explicit matrix form
\begin{equation}
 \scalemath{0.83}{
\Box_0=\begin{pmatrix}
\Box & 0&..&0\\
0& \Box&..& 0\\
..&..&..&..\\
0& 0&..&\Box\end{pmatrix} + \begin{pmatrix}
\Box_{1,1} & \Box_{1,2}&..&\Box_{1,m}\\
\Box_{2,1}& \Box_{2,2}&..& \Box_{2,m}\\
..&..&..&..\\
\Box_{k,1}& \Box_{k,2}&..&\Box_{k,m}\end{pmatrix}},
\end{equation}

where $\Box_{i,j}:A\rightarrow A$ are scalar operators of order $k-1,$ so are given by
$\Box_{i,j}:=\sum_{|\sigma|\leq k-1}A_{ij}^{\sigma}\partial_{\sigma}$ for functions $A_{ij}^{\sigma}\in A.$
The matrix operator $\tilde{\Box}:=||\Box_{i,j}||:P\rightarrow P$ acts on an element $p=p^{\alpha}e_{\alpha}\in P$ as 
$\tilde{\Box}(p)=\sum_{\beta}^m\sum_{\alpha=1}^{m}\sum_{|\sigma|\leq k}A_{\beta\alpha}^{\sigma}\partial_{\sigma}(p^{\alpha})e_{\beta}.$

\begin{rmk}
When $k=1$, since $\text{Diff}_0(P,P)=\text{End}_{\mathbb{K}}(P),$ one recovers the local coordinate description of a Der-operator $X_0\in D(\mathscr{A})_0 \cong \text{Der}(P).$ 
\end{rmk}

We now compute explicit formulas for the Lie algebra structure on diolic differential operators. Let us first determine the homogeneous degree zero components,
$\big[.,.\big]:\text{Diff}_{k}(\mathscr{A})_0\times \text{Diff}_{\ell}(\mathscr{A})_0\rightarrow \text{Diff}_{k+\ell-1}(\mathscr{A})_0.$
In particular, for $\Box_0=\pmb{\Box}_{0}^A+||\Box_{i,j}||\in\text{Diff}_{k}(\mathscr{A})_0,$ and for $\nabla_0=\pmb{\nabla}_{0}^A+||\nabla_{i,j}||\in\text{Diff}_{\ell}(\mathscr{A})_0,$ we find, using the alternative notation for the matrices $\Box_0^A\mathbb{I}:=\pmb{\Box}_0^A$ where $\mathbb{I}$ is the identity \footnote{The notation $\Box_0^A\mathbb{I}$ does not mean the matrix with entries $\text{diag}(\Box_0^A(1),..,\Box_0^A(1)),$ but rather that with entries $\text{diag}(\Box_0^A,..,\Box_0^A)$.}
\begin{equation}
\big[\Box_0,\nabla_0\big]=\underbrace{\big(\big[\Box_{0}^A,\nabla_{0}^A\big]\mathbb{I}\big)}_{\mathscr{O}\leq k+\ell-1}+\underbrace{\big|\big|[\Box_{0}^A,\nabla_{i,j}]\big|\big|}_{\mathscr{O}\leq k+\ell-2}
-\underbrace{\big|\big|[\nabla_{0}^A,\Box_{i,j}]\big|\big|}_{\mathscr{O}\leq k+\ell-2}+\underbrace{\big[||\Box_{i,j}||,||\nabla_{i,j}||\big]}_{\mathscr{O}\leq k+\ell-2}.
\end{equation}
The right hand side is again the sum of two differential operators, the first of which has order $k+\ell-1$ and the second of order $(k+\ell-1)-1.$ This is consistent with the splitting of degree zero diolic differential operators, so indeed the commutator of two such degree zero operators is again a degree zero operator. 

\begin{rmk}
It is worth noticing that if the module $P$ is one dimensional, then the final commutator will be of order $\leq (k-1)+(\ell-1)-1.$
\end{rmk}

Turning our attention to degree $1$ diolic differential operators, we have $\Box_1\in\text{Diff}_{k}(\mathscr{A})_1\cong\text{Diff}_{k}(A,P),$ we have
\begin{equation*}
\Box_1:A\rightarrow P, \hspace{1mm}
 f\longmapsto \Box_1^A(f):=\big(
\overline{\Box}_1^A(f),
\overline{\Box}_2^A(f),..,
\overline{\Box}_m^{A}(f)\big)^T
\end{equation*}
where $\overline{\Box}^{\alpha}\in\text{Diff}_k(A,A)$ which will be specified in coordinates as
$\overline{\Box}^{\alpha}(f):=\sum_{|\sigma|\leq k}B_{\sigma}^{\alpha}\partial^{\sigma}(f).$ Note that $\overline{B}_j^{\sigma}\in A$ are coefficient functions determining the scalar differential operator $\overline{\Box}_j^A.$
We can write the entire operator $\Box_1$, acting on $f\in A$ as
$\Box_1(f)=\sum_{\alpha=1}^m\overline{\Box}^{\alpha}(f)e_{\alpha}=\sum_{\alpha=1}^m\sum_{|\sigma|\leq k}B_{\sigma}^{\alpha}\partial^{\sigma}(f)e_{\alpha}.$
From this, it is evident that the operator $\Box_1\in\text{Diff}_k(A,P)$ is determined by the $\text{rank}(P)$=tuple of functions $(B_{\sigma}^1,...,B_{\sigma}^{rk(P)}).$
We can readily verify the commutator between a degree $1$ and a degree $0$ graded diolic differential operator, acting on an element $a\in A,$ as 
\begin{equation}
\label{eqn:Degzeroonecommutator}
 \scalemath{0.85}{
[\Box_0,\Delta_1]^{(gr)}(a)=\begin{pmatrix}
[\Box_{0}^A,\overline{\Delta}_{1}^A](a)\\
.\\
.\\
[\Box_{0}^A,\overline{\Delta}_m^A](a)
\end{pmatrix}+
\begin{pmatrix}
\Box_{1,1} & \Box_{1,2}&..&\Box_{1,m}\\
\Box_{2,1}& \Box_{2,2}&..& \Box_{2,m}\\
..&..&..&..\\
\Box_{k,1}& \Box_{m,2}&..&\Box_{k,m}\end{pmatrix}\begin{pmatrix}
\overline{\Delta}_1^A(a)\\
.\\
.\\
\overline{\Delta}_m^A(a)
\end{pmatrix}=\begin{pmatrix}
[\Box_{0}^A,\overline{\Delta}_{1}^A](a)\\
.\\
.\\
[\Box_{0}^A,\overline{\Delta}_m^A](a)
\end{pmatrix}+\begin{pmatrix}
\sum_{\beta=1}^m\Box_{1,\beta}\big(\overline{\Delta}_{\beta}^A(a)\big)
\\
.
\\
.
\\
\sum_{\beta=1}^m\Box_{m,\beta}\big(\overline{\Delta}_{\beta}^A(a)\big)
\end{pmatrix}}.
\end{equation}
\vspace{1mm}

Consequently, the above $P$-valued operator determined by the graded commutator is a the differential operator $A\rightarrow P$ of order $k+\ell-1.$
\begin{rmk}
It is of order $k+\ell-1$ since $\Box_0^A$ is of order $k$, $\overline{\Delta}_1^A$ are scalar of order $\ell$ so their commutator arising in the first factor is of order $k+\ell-1,$ while the operators $\Box_{1,\beta}$ in the secon factor are of order $k-1$. Thus the composition of the two scalar operators $\Box_{j,\beta}\circ \overline{\Delta}_{\beta}^A$ will be another scalar operator of order $(k-1)+\ell.$
\end{rmk}

\section{The Diolic Hamiltonian formalism: Diolic symbols}
Consider the $k$'th order diolic Atiyah sequence (\ref{eqn:diAtk}).
We saw above that this sequence is filtered in the sense that we have sub-sequences for all $\ell\leq k,$ induced from the natural order filtration in differential operators. Consequently, we may consider the corresponding quotient sequences $\text{At}_{k-1}(\mathscr{A})\hookrightarrow \text{At}_k(\mathscr{A})\rightarrow \text{At}_k(\mathscr{A})/\text{At}_{k-1}(\mathscr{A}).$ 
Owing to the existence of these order filtrations, by passing to quotients we find
$$0\rightarrow \frac{\text{Diff}_{k-1}(P,P)}{\text{Diff}_{k-2}(P,P)}\rightarrow \frac{\text{Diff}_{k}(\mathscr{A},\mathscr{A})_0}{\text{Diff}_{k-1}(\mathscr{A},\mathscr{A})_0}\rightarrow\frac{\text{Diff}_{k+1}(A)}{\text{Diff}_{k}(A)}\rightarrow 0,$$
for each $k\geq 0,$ (where we set $\text{Diff}_{-k}:=0$)
which by definition of the algebra of symbols, yields the corresponding $k$'th order \emph{diolic symbol sequence},
\begin{equation}
0\rightarrow\text{Smbl}_{k-1}(P,P)\rightarrow \text{Smbl}_{k}(\mathscr{A},\mathscr{A})_0\rightarrow \text{Smbl}_k(A)\rightarrow 0.
\end{equation}

In this way we see that $\text{Smbl}_{k,0}(\mathscr{A})$ is a `compound' constructed from $\text{Smbl}_k(A)$, which is the usual Hamiltonian formalism, and $\text{Smbl}_{k-1}(P,P).$ 
Moreover, if $P$ is one dimensional, then $\text{Smbl}(P,P)$ is commutative.
We have the following result.

\begin{lem}
Let $\mathscr{A}$ be a diolic algebra with $P$ a projective module of finite rank with fixed splittings of all diolic Atiyah sequences. Then $\mathrm{Smbl}_{k}(\mathscr{A},\mathscr{A})_0\cong\mathrm{Smbl}_k(A,A)\oplus \mathrm{Smbl}_{k-1}(P,P),$ for all $k$. 
\end{lem}

Next, let us observe that generally, for $\Delta\in\text{Diff}(Q,Q)$ such that $[\Delta]\in\text{Smbl}(Q,Q),$ for some order, and $[\nabla]\in\text{Smbl}(P,P),$ we have for $[\Box]\in\text{Smbl}(P,Q)$ two well-defined actions,
$$[\Delta]\star^<[\Box]:=[\Delta\circ \Box], [\nabla]\star^>[\Box]:=[\Box\circ\nabla],$$
where the $\star$ operation is the obvious one defined as in \ref{symbolalgebrastructure}.
In other words, we obtain a left $\text{Smbl}(Q,Q)$ and right $\text{Smbl}(P,P)$-module structure in $\text{Smbl}(P,Q).$ 
Consequently, $\text{Smbl}(A,P)$ is indeed a module over both $\text{Smbl}(A,A)$ and $\text{Smbl}(P,P)$ in a well-defined manner. Therefore, $\text{Smbl}(\mathscr{A})_1$ is indeed a $\text{Smbl}(\mathscr{A})_0$-module.
As we found above, we have no degree $g>1$ components of the module $\text{Diff}(\mathscr{A},\mathscr{A}),$ so we immediately find a diolic version of lemma \ref{symbolalg}.
\begin{thm}
The algebra of diolic symbols $\mathrm{Smbl}(\mathscr{A})_{\mathcal{G}}$ is a diolic algebra.
\end{thm}
\subsubsection{A characterization of Diolic symbols of degree $0$}
We now give a useful characterization of degree $0$-diolic symbols in the geometric case where $C^{\infty}(M)\oplus \Gamma(\pi)$ for some rank $m$ bundle $E\rightarrow M$.
We need a standard preliminary result which can be found in \cite{KraVerb}.
\begin{lem}
\label{symsmbl}
Let $A$ be a smooth algebra and $P,Q$ be projective $A$-module. Then
 If $A$ is an algebra over the field of rational numbers, then we have for each $k\geq 0$, that  $\mathrm{Smbl}_k(P,Q)\cong S^k\big(D_1(A)\big)\otimes_A \mathrm{Hom}_A(P,Q).$
 Thus we interpret these as symmetric $k$-multiderivations of $A$ with values in $\mathrm{Hom}_A(P,Q).$ In particular, $\mathrm{Smbl}_*(A,P)\cong P\otimes_A \mathrm{Smbl}_*(A),$ while the algebra of symbols itself enjoys the following description
$\mathrm{Smbl}_*(A)\cong \mathrm{Sym}^*\big(\mathrm{Smbl}_1(A)\big)=\mathrm{Sym}^*\big(D_1(A)\big).$
 \end{lem}

This generalizes straightforward to the setting of graded geometry or even the differential graded, or $\mathbb{N}Q$-geometric setting.
\begin{exm}
Let $\mathscr{A}$ be the $DG$ algebra of smooth functions on a $DG$ manifold $\mathscr{M}$ with homological vector field $d$, and consider $D(\mathscr{A})$ with $\partial:D(\mathscr{A})\rightarrow D(\mathscr{A}),$ the derivation $[d,-].$
The universal enveloping algebra, $\mathcal{U}\big(D(\mathscr{A})\big),$ identifies with $\mathrm{Diff}(\mathscr{A},\mathscr{A}).$ The differential is again given by $[d,-].$ It follows that $\text{Gr}\big(\mathcal{U}\big(D(\mathscr{A})\big)\big)$
identifies with $\mathrm{Sym}_{\mathscr{A}}^{\bullet}\big(D(\mathscr{A})\big),$
which is precisely the the DG Poisson algebra of fiber-wise polynomial functions on $T^*\mathscr{M}.$
\end{exm}
Returning to the diolic formalism, in particular, the degree $1$ symbols, lemma \ref{symsmbl} tells us the following. For $\Delta$ a degree $1$ diolic operator with symbol $\text{smbl}_k(\Delta),$ under the isomorphism above which is defined by setting
$\gamma_k^P\big(\text{smbl}_k(\Delta)\big)(da_1\cdot..\cdot da_k):=\big(\delta_{a_1}\circ ...\circ\delta_{a_k}\big)(\Delta),$
where $d a_i\cdot da_j$ denotes the multiplication in $\text{Sym}^k(\Lambda^1),$ we have an isomorphism of $A$-modules  
$\mu_P:P\otimes_A \text{Sym}^k\big(D(A)\big)\rightarrow \text{Smbl}_k(A,P),p\otimes X_1\odot X_2\odot..\odot X_k\longmapsto \frac{1}{k!}\text{smbl}_k\big(1_p\circ X_1\circ X_2\circ...\circ X_k\big),$
where $1_p\in \text{Hom}_A(A,P)$ is the degree zero operator of multiplication $1_p(a):=ap,$
is such that $\gamma_P\circ \mu_P=id.$ 
\begin{rmk}
This description is to be interpreted as  characterizing the degree $1$ functions on the diolic phase space $T^*\mathscr{A}.$
\end{rmk}

Let us turn our attention to finding an analogous characterization in the degree zero case. To this end, consider the following diagram with short exact columns, where each row is the corresponding diolic Atiyah sequence

\begin{equation}
\begin{tikzcd}
0 \arrow[r, ]&\text{Diff}_{k-1}\big(P,P\big)\arrow[d, ""] \arrow[r, ""]
&\text{Diff}_{k}(\mathscr{A})_0\arrow[d, "" ] \arrow[r, ""]  &\text{Diff}_k(A) \arrow[d, ""] \arrow[r, ""]&  0\\
0 \arrow[r, ]&\text{Diff}_{k}\big(P,P\big)\arrow[d, ""] \arrow[r, ""]
&\text{Diff}_{k+1}(\mathscr{A})_0\arrow[d, "" ] \arrow[r, ""]  &\text{Diff}_{k+1}(A)\arrow[d, ""] \arrow[r, ""]&  0 \\
0 \arrow[r, ""] & \text{Smbl}_k(P,P) \arrow[r, ""] & \text{Smbl}_{k+1}(\mathscr{A})_0\arrow[r, ""] & \text{Smbl}_{k+1}(A)\arrow[r, ""] & 0, 
\end{tikzcd}
\end{equation}
To completely characterize the degree zero graded symbols of order $k$,
we require a surjective morphism $\lambda^{k+1}:\text{Diff}_{k+1}(\mathscr{A})_0\rightarrow \text{Smbl}_{k+1}(\mathscr{A})_0$ such that $\ker(\lambda^{k+1})\cong \text{Diff}_{k}(\mathscr{A})_0.$
Using lemma \ref{symsmbl}, we rewrite the bottom row as
\begin{equation}
    \label{eqn:degzerosymbol}
0\rightarrow \text{Sym}^{k}\big(D(A)\big)\otimes \text{End}(P)\rightarrow \mathcal{S}^{k+1}(\mathscr{A})\rightarrow \text{Sym}^{k+1}\big(D(A)\big)\rightarrow 0,
\end{equation}
where we write $\mathcal{S}^{k+1}(\mathscr{A})$ as the would-be placeholder under the equivalence of symbols with tensor fields. The following characterizes this module in the spirit of lemma \ref{symsmbl}.
\newpage 

\begin{thm}
$\mathcal{S}^k(\mathscr{A})=\{\Delta:A\times...\times A\rightarrow\mathrm{Der}(P)\}$ coincides with $\mathrm{Der}(P)$-valued symmetric $(k-1)$-derivations with symmetric multi-symbols.
This means $\Delta(f_1,..,f_{k-1},gp)=\sigma_{\Delta}(f_1,..,f_{k-1},g)p+g\Delta(f_1,..,f_{k-1},p).$
\end{thm}
\begin{proof}
Let us construct the morphism $\lambda^k:\text{Diff}_{k}(\mathscr{A},\mathscr{A})_0\rightarrow \mathcal{S}^k(\mathscr{A}).$ We define this morphism for $\Delta_0=(\Delta_A,\Delta_P)$ such that $\delta_a^k(\Delta_A)=\delta_a^k(\Delta_P),$ as
$\lambda^k(\Delta_0):=\mathcal{P}_{\Delta},$
where
$\mathcal{P}_{\Delta}(a_1,..,a_{k-1}):=\big(\delta_{a_1}\circ...\circ\delta_{a_{k-1}}\big)\Delta_P\in\text{Der}(P).$
It enjoys the following relation
$\mathcal{P}_{\Delta}(a_1,..,a_{k-1})(ap)-a\mathcal{P}_{\Delta}(a_1,..,a_{k-1})(p)=\big(\delta_a\delta_{a_1}..\delta_{a_{k-1}}\Delta_P\big)(p).$
We can see that if $\Delta$ is of order $k-1$ then the right hand side necessarily vanishes and we have that this morphism descends to quotients (ie. symbols) because, by definition of $\lambda^k$, this operation `kills' all parts in levels $k-1$ and lower.

We want to show that $\ker(\lambda^k)\cong\text{Diff}_{k-1}(\mathscr{A})_0.$ Suppose that $\Delta_0\in\text{ker}(\lambda^k).$ Then we have $\lambda^k(\Delta_0)=\mathcal{P}_{\Delta}=0,$ and consequently
$$\lambda^k(\Delta_0)(a_1,..,a_{k-1},bp)=\lambda^k(\Delta_0)(a_1,...,a_{k-1})(bp)=\big(\delta_{a_1},..,\delta_{a_{k-1}}\Delta_P)(bp)-\big(\delta_{a_1}...\delta_{a_{k-1}}\Delta_A)(1_A)(bp),$$
where in the last term we remark that $(\delta_{a_1},..,\delta_{a_{k-1}}\Delta_A)(1_A)$ is the left multiplication of an element of $A$ on the element $bp\in P,$ as any such operator is determined by its value on $1_A.$ Now since $\Delta_0$ is in the kernel of $\lambda^k,$ by assumption, we have 
$$\delta_{a_1}..\delta_{a_{k-1}}\Delta_P=\big(\delta_{a_1}..\delta_{a_{k-1}}\Delta_A)(1_A).$$
But this means that $(\delta_a..\delta_{a_{k-1}}\Delta_P)(p)=0,$ for all $p\in P$ by the above relation. Thus this is just stating that $\Delta_P\in\text{Diff}_{k-1}(P,P).$ Applying $\delta_a$ to the relation directly above, we find that 
$\Delta_{A}\in\text{Diff}_{k-1}(A).$ Moreover, we find the condition $\delta_a^{k-1}\Delta_A=\delta_a^{k-1}\Delta_P.$ Consequently the pair $(\Delta_A,\Delta_P)$ defines an element in $\text{Diff}_{k-1}(\mathscr{A})_0.$

So in fact, we have demonstrated that the column $\text{Diff}_{k-1}(\mathscr{A})_0\rightarrow \text{Diff}_{k}(\mathscr{A})_0\rightarrow \text{Smbl}_{k}(\mathscr{A})_0$ of the above diagram is isomorphic to 
$\ker(\lambda^k)\rightarrow \text{Diff}_{k}(\mathscr{A})_0\rightarrow \mathcal{S}^k(\mathscr{A}),$ and by definition of $\lambda^k,$ it descends to quotients and we have the isomorphism 
$\text{Smbl}_{k}(\mathscr{A})_0\cong \mathcal{S}^k(\mathscr{A}).$
\end{proof}
The final aspect of the diolic symbol algebra we wish to discuss is the canonical Poisson bracket; however, as this is immediate from sub-section \ref{sssec:Coords}, we mention it only briefly.
\subsubsection{The Diolic Poisson bracket}
Making use of the explicit computations for the commutators of diolic differential operators given above, in particular the computation (\ref{eqn:Degzeroonecommutator}), we may determine the local expression of the diolic Poisson bracket. For instance, the diolic Poisson bracket determined by the commutator $[\Box_0,\Delta_1]$ is computed as
\begin{eqnarray*}
\big\{\text{smbl}_{k,0}(\Box_0),\text{smbl}_{\ell,1}(\Delta_1)\big\}
&=&\begin{pmatrix}
\text{smbl}_{k+\ell-1}\big([\Box_0^A,\overline{\Delta}_1^A]\big)
\\
.
\\
.
\\
\text{smbl}_{k+\ell-1}\big([\Box_0^A,\overline{\Delta}_m^A]\big)
\end{pmatrix}+\begin{pmatrix}
\sum_{\beta=1}^m\text{smbl}_{k+\ell-1}\big(\Box_{1,\beta}\circ \overline{\Delta}_{\beta}^A\big)
\\
.
\\
.
\\
\sum_{\beta=1}^m\text{smbl}_{k+\ell-1}\big(\Box_{m,\beta}\circ \overline{\Delta}_{\beta}^A\big)
\end{pmatrix}
\\
&=&
\begin{pmatrix}
\big\{\text{smbl}_{k}(\Box_0^A),\text{smbl}_{\ell}(\overline{\Delta}_1^A)\big\}
\\
.
\\
.
\\
\big\{\text{smbl}_{k}(\Box_0^A),\text{smbl}_{\ell}(\overline{\Delta}_m^A)\big\}
\end{pmatrix}+\begin{pmatrix}
\sum_{\beta=1}^m\text{smbl}_{k-1}(\Box_{1,\beta}) \star \text{smbl}_{\ell}(\overline{\Delta}_{\beta}^A)
\\
.
\\
.
\\
\sum_{\beta=1}^m\text{smbl}_{k-1}(\Box_{m,\beta})\star \text{smbl}_{\ell}(\overline{\Delta}_{\beta}^A)
\end{pmatrix}.
\end{eqnarray*}
In the last line we use the definition of the (ungraded) Poisson bracket of symbols, as well as the definition of the commutative algebra multiplication $\star$ on the commutative $A$-algebra $\text{Smbl}(A,A).$
The remaining graded components of this bracket are similarly established.
\begin{rmk}
This result shows that the graded Poisson bracket of a degree zero and degree $1$ diolic symbol is determined entirely by the Poisson bracket $\{-,-\}$ on the ungraded algebra of symbols $\text{Smbl}(A)$ as well as the commutative algebra multiplication $\star$ on this algebra, in the basis defined by $P$.
\end{rmk}

\section*{Concluding remarks}
In this work we provided an introduction to a novel purely algebraic formalism for studying traditional differential geometry of vector bundles in the language of diolic differential calculus. Along the way we indicated several possible interesting generalizations for more general graded structures, for instance, triolic algebras. The study of these will appear in a sequel work and in the authors PhD thesis \cite{Kry}. 

Building on the results obtained in this work, one may further explore the diolic formalism by understanding what are quantizations of the diole algebra, in the sense of \cite{Lych}. This is straightforwardly achieved using our description of diolic differential operators and their symbols and the discussion of these aspects is treated in a short note soon to appear.

\bibliographystyle{alpha}
\bibliography{Bibliography.bib}

\begin{thebibliography}{DPV09}

\bibitem[Ati57]{Ati}
Michael Atiyah.
\newblock Complex analytic connections in fibre bundles.
\newblock {\em Transactions of the American Mathematical Society},
  85(1):181--207, 1957.

\bibitem[Bru17]{BruzzoLie}
Ugo Bruzzo.
\newblock Lie algebroid cohomology as a derived functor.
\newblock {\em Journal of Algebra}, 483:245--261, 2017.

\bibitem[DPV09]{Fat}
Alessandro De~Paris and Alexandre Vinogradov.
\newblock {\em Fat manifolds and linear connections}.
\newblock World Scientific, 2009.

\bibitem[KK95]{KerKra02}
Joseph Krasil’shchik and Paul Kersten.
\newblock Graded differential equations and their deformations: a computational
  theory for recursion operators.
\newblock In {\em Geometric and Algebraic Structures in Differential
  Equations}, pages 167--191. Springer, 1995.

\bibitem[KK13]{KerKra01}
Joseph Krasil'shchik and Paul Kersten.
\newblock {\em Symmetries and recursion operators for classical and
  supersymmetric differential equations}, volume 507.
\newblock Springer Science \& Business Media, 2013.

\bibitem[Kra97]{Kra01}
Joseph Krasil’shchik.
\newblock Calculus over commutative algebras: a concise user guide.
\newblock {\em Acta Applicandae Mathematica}, 49(3):235--248, 1997.

\bibitem[Kry20]{Kry}
Jacob Kryczka.
\newblock Cohomological aspects of the hamiltonian formalism: steps towards
  quantum observability.
\newblock {\em Doctoral thesis (To appear), University of Angers, France},
  2020.

\bibitem[KV98]{KraVerb}
Joseph Krasil'shchik and Alexander Verbovetsky.
\newblock {\em Homological methods in equations of mathematical physics}.
\newblock Advanced Texts in Mathematics, Open Education \& Sciences, Opava,
  1998.

\bibitem[Lyc99]{Lych}
Valentin Lychagin.
\newblock Quantum mechanics on manifolds.
\newblock {\em Acta Applicandae Mathematica}, 56(2-3):231--251, 1999.

\bibitem[Nes06]{Nes}
Jet Nestruev.
\newblock {\em Smooth manifolds and observables}, volume 220.
\newblock Springer Science \& Business Media, 2006.

\bibitem[PV13]{PorVez}
Mauro Porta and Gabriele Vezzosi.
\newblock Infinitesimal and square-zero extensions of simplicial algebras.
\newblock {\em arXiv preprint arXiv:1310.3573}, 2013.

\bibitem[Rub80]{Rub01}
Vladimir Rubtsov.
\newblock The cohomology of the der complex.
\newblock {\em Russian Mathematical Surveys}, 35(4):190, 1980.

\bibitem[Ver96]{Ver}
Alexander Verbovetsky.
\newblock Lagrangian formalism over graded algebras.
\newblock {\em Journal of Geometry and Physics}, 18(3):195--214, 1996.

\bibitem[Vin72]{Vin01}
Alexandre Vinogradov.
\newblock The logic algebra for the theory of linear differential operators.
\newblock In {\em Doklady Akademii Nauk}, volume 205, pages 1025--1028. Russian
  Academy of Sciences, 1972.

\bibitem[Vin86]{Vin04}
Alexandre Vinogradov.
\newblock Geometric singularities of solutions of nonlinear partial
  differential equations.
\newblock {\em Differential geometry and its applications (Brno, 1986)},
  27:359--379, 1986.

\bibitem[Vin94]{Vin02}
Alexandre Vinogradov.
\newblock From symmetries of partial differential equations towards secondary
  (“quantized”) calculus.
\newblock {\em Journal of Geometry and Physics}, 14(2):146--194, 1994.

\bibitem[Vin13]{Vin03}
Alexandre Vinogradov.
\newblock What are symmetries of nonlinear pdes and what are they themselves?
\newblock {\em arXiv preprint arXiv:1308.5861}, 2013.

\bibitem[VK75]{KraVin}
Alexandre Vinogradov and Joseph Krasil'shchik.
\newblock What is the hamiltonian formalism?
\newblock {\em Russian Mathematical Surveys}, 30(1):177, 1975.

\bibitem[VV95]{VezVin}
Gabriele Vezzosi and Alexandere Vinogradov.
\newblock On higher analogs of the de rham complex.
\newblock {\em Preprint ESI}, 202, 1995.

\end{thebibliography}

\end{document}